\documentclass[a4paper,12pt,english]{amsart}

\usepackage{amsmath,upref,amssymb,amsthm,amsfonts}
\usepackage{tensor}
\usepackage{color}
\input xypic
\CompileMatrices

\RequirePackage{calrsfs}
\DeclareSymbolFont{rsfscript}{OMS}{rsfs}{m}{b}
\DeclareSymbolFontAlphabet{\mathrsfs}{rsfscript}
\renewcommand{\mathcal}{\mathrsfs}

\newcommand{\DMO}{\DeclareMathOperator}

\numberwithin{equation}{subsection}

\swapnumbers

\theoremstyle{plain} % default
\newtheorem{theorem}[subsection]{Theorem}
\newtheorem{prop}[subsection]{Proposition}
\newtheorem{lemma}[subsection]{Lemma}
\newtheorem{corollary}[subsection]{Corollary}

\theoremstyle{definition}
\newtheorem{definition}[subsection]{Definition}

\theoremstyle{remark}
\newtheorem{remark}[subsection]{Remarks}

\theoremstyle{remark}
\newtheorem{rem}[subsection]{Remark}

\newtheorem*{case*}{Case}

\makeatletter
\renewcommand{\p@enumi}{\thesubsection}
\renewcommand{\p@enumii}{\thesubsection\theenumi}
\makeatother

\newenvironment{num}{\renewcommand{\theenumi}{(\alph{enumi})}
                      
                                          \begin{enumerate} }
                    {\end{enumerate} }
 \newenvironment{conds}{\renewcommand{\theenumi}{(\roman{enumi})}
                       
                        \begin{enumerate} }
                     {\end{enumerate} }
 \newenvironment{subconds}{
                       
                        \begin{enumerate} }
                     {\end{enumerate} }
\newcommand{\bee}{\begin{equation*}}
\newcommand{\eee}{\end{equation*}}

\newcommand{\bs}{\begin{split}}
\newcommand{\es}{\end{split}}

\newcommand{\bc}{\begin{cases}}
\newcommand{\ec}{\end{cases}}

\newcommand{\bml}{\begin{multline}}
\newcommand{\eml}{\end{multline}}

\newcommand{\bmll}{\begin{multline*}}
\newcommand{\emll}{\end{multline*}}

\newcommand{\belb}[1]{\mar{#1}\begin{equation}\label{#1}}
\newcommand{\mpair}[1]{\pair{\,#1\,}}

\newcommand{\mset}[1]{\set{\,#1\,}}

\newcommand{\pair}[1]{\langle #1\rangle}

\newcommand{\set}[1]{\{#1\}}

\newcommand{\CC}{{\mathcal C}}

\newcommand{\CF}{{\mathcal F}}
\newcommand{\CG}{{\mathcal G}}

\newcommand{\CI}{{\mathcal I }}

\newcommand{\CP}{{\mathcal P}}

\newcommand{\CR}{{\mathcal R}}
\newcommand{\CS}{{\mathcal S}}
\newcommand{\CT}{{\mathcal T}}
\newcommand{\CU}{{\mathcal U}}

\newcommand{\CX}{{\mathcal X}}

\newcommand{\bbR}{{\mathbb R}}

\newcommand{\bbS}{{\mathbb S}}
\newcommand{\bbB}{{\mathbb B}}

\def\a{\alpha}
\def\b{\beta}
\def\g{\gamma}
\def\G{\Gamma}

\def\D{\Delta}

\def\l{\lambda}

\def\L{\Lambda}

\def\o{\omega}
\def\O{\Omega}
\def\r{\rho}

\def\s{\sigma}

\def\th{\theta}

\def\t{\tau}

\def\to{\rightarrow}

\def\lr{\longrightarrow}

\def\lsupp#1#2{\kern\scriptspace\vphantom{#2}^{#1}\kern-\scriptspace#2}

\def\lsubb#1#2{\kern\scriptspace\vphantom{#2}_{#1}\kern-\scriptspace#2}

\newcommand{\ssect}{\subsection}

\newcommand{\seq}{{\,\subseteq\,}}

\newcommand{\sneq}{{\,\subsetneq\,}}

\newcommand{\sreq}{{\,\supseteq\,}}

\newcommand{\sm}{{\,\setminus\,}}

\newcommand{\eset}{{\emptyset}}

\newcommand{\ck}[1]{{#1}^\vee}

\newcommand{\wh}{\widehat}

\newcommand{\ol}{\overline}

\DeclareMathOperator{\linspan}{{\mathrm{span}}}

\def\disjun{\sqcup}

\def\bdisjun{\bigsqcup}

\DeclareMathOperator{\supp}{{\mathrm{supp}}}

\DMO{\Supp}{{\mathrm{Supp}}}

\DeclareMathOperator{\Stab}{{\mathrm{Stab}}}

\DeclareMathOperator{\Ind}{{\mathrm{Ind}}}

\DMO{\Rad}{{\mathrm{Rad}}}

\DMO{\Ann}{{\mathrm{Ann}}}

\newcommand{\Id}{\mathrm{Id}}

\DMO {\Add}{{\mathrm{Add}}}

\DMO {\add}{{\mathrm{add}}}

\DMO {\Img}{{\mathrm{Im}}}

\DMO {\coim}{{\mathrm{Coim}}}

\DMO {\coker}{{\mathrm{Coker}}}

\DMO {\colim}{\varinjlim}

\DMO {\plim}{\varprojlim}

\DMO {\mEnd}{{\mathrm{End}}}

\DMO {\mend}{{\mathrm{end}}}

\DMO {\Proj}{{\mathrm{Proj}}}

\DMO {\Ext}{{\mathrm{Ex}t}}

\DMO {\ext}{{\mathrm{ext}}}

\DMO {\tor}{{\mathrm{tor}}}

\DMO {\Tor}{{\mathrm{Tor}}}

\DMO {\Hom}{{\mathrm{Hom}}}

\DMO {\HOM}{{\mathrm{HOM}}}

\DMO {\Modfg}{{\mathrm{-Modfg}}}

\DMO {\modulfg}{{\mathrm{-modfg}}}

\DMO {\modul}{{\mathrm{-mod}}}

\DMO {\Mod}{{\mathrm{-Mod}}}

\DMO {\pdim}{{\mathrm{proj.dim.}}}

\DMO {\PD}{{{\mathrm{Proj.Dim.}}}}

\DMO {\gldim}{{\mathrm{gr.gl.dim.\ }}}

\DMO {\grad}{{\mathrm{rad} }}

\DMO {\Sheaves}{{\mathrm{Sh}}}

\DMO {\Flab}{{\mathrm{Fl}}}

\DMO {\Poinc}{{\mathrm{Poinc}}}

\DMO {\Groth}{{{{\mathrm{K}}}_0}}

\DMO {\Mat}{{\mathrm{Mat}}}

\DMO \Sl{{\mathrm{sl}}}

\DMO \SL{{\mathrm{SL}}}

\DMO \Gl{{\mathrm{gl}}}

\DMO \GL{{\mathrm{GL}}}

\DMO \lcm{\mathrm{{lcm}}}

\DMO \rank{{\mathrm{rank}}}

\DMO \diag{{\mathrm{diag}}}

\DMO\vtx{{\mathrm{Vert}}}

\DMO\pc{{\mathrm{ParClos}}}

\newcommand{\bv}{{\mathbf{v}}}
\newcommand{\bu}{{\mathbf{u}}}
\newcommand{\ba}{{\mathbf{a}}}
\newcommand{\bd}{{\mathbf{d}}}
\newcommand{\bx}{{\mathbf{c}}}
\newcommand{\bb}{{\mathbf{b}}}
\newcommand{\bI}{{\mathbf{I}}}
\newcommand{\bJ}{{\mathbf{J}}}

\newcommand{\bw}{{\mathbf{w}}}

\newcommand{\bsy}{\boldsymbol}

\newcommand{\Int}{{\mathbb Z}}
\newcommand{\Nat}{{\mathbb N}}
\newcommand{\real}{{\mathbb R}}

\DMO{\ob}{ob}
\DMO{\mor}{mor}
\DMO{\tr}{tr}
\DMO{\spec}{spec}

\DMO{\cov}{cov}
\DMO{\Sym}{Sym}
\DMO{\Dih}{Dih}
\DMO{\Spec}{Spec}
\DMO{\domn}{dom}
\DMO{\cod}{cod}
\DMO{\ord}{ord}

\newcommand{\inv}{^{-1}}

\DMO{\stab}{{\mathrm{Stab}}}
\DMO{\inter}{{\mathrm{Int}}}
\DMO{\coeff}{{\mathrm{Coeff}}}
\DMO{\conv}{{\mathrm{Conv}}}
\begin{document}

\title{Geometry of certain finite Coxeter group actions}

\author{M. J.  Dyer and G. I.  Lehrer}
\address{Department of Mathematics 
\\ 255 Hurley Building\\ University of Notre Dame \\
Notre Dame, Indiana 46556, U.S.A.}
\email{dyer.1@nd.edu}
\address{Department of Mathematics, University of Sydney, Sydney. NSW. 2006}
\email{gustav.lehrer@sydney.edu.au}
%\dedicatory{}

%\thanks{}
%\keywords{}

%\date{\today}
\begin{abstract}  We determine a fundamental domain for the diagonal action
of a finite Coxeter group $W$ on $V^{\oplus n}$, where $V$ is the reflection representation.
This is used to give a stratification of $V^{\oplus n}$, which is respected by the group
action, and we study the geometry, topology and combinatorics of this 
stratification. These ideas are used to obtain results on the classification of root 
subsystems up to conjugacy, as well as a character formula for $W$. 
\end{abstract}

\maketitle
%%%%%%%%%%%%%%%%%%%%%%%%%%%%%%%%
%
%                          
%
%
%%%%%%%%%%%%%%%%%%%%%%%%%%%%%%%%
\section*{Introduction}\label{s0}
Let $\Phi$ be a finite root system in the Euclidean space $V=\bbR^\ell$, whose
inner product we shall denote $\mpair{-,-}$. Let $W=W(\Phi)$ be the corresponding
Weyl group. This is a finite reflection group on $V$, and the choice of a simple 
subsystem of $\Phi$ defines a fundamental region, known as the fundamental chamber,
for the action of $W$ on $V$.

The group $W$ acts diagonally on $V^n:=V^{\oplus n}=V\oplus\dots\oplus V$, and  
our first main result, Theorem \ref{thm:fundreg} in \S\ref{s1} below,
 is the determination of a fundamental region $\CC_W^{(n)}$ for 
this $W$-action. This turns out to be a locally
closed, convex subset of $V^n$.
In \S\ref{sec:strat} we show how our fundamental region may be used to obtain a stratification 
of $V^n$ by locally closed subsets which are transformed into each other under the
$W$-action, and are such that the closure of any one is a union of such sets.
This leads to a combinatorial structure which significantly generalises the Coxeter complex
of $W$. We study the topology and combinatorics of this stratification in \S\ref{s:minmax}. 

As applications of these results, we give in \S\ref{ss:chars} a character formula for $W$, which generalises the usual 
Solomon formula for the alternating character of $W$.

Then, in \S\ref{sec:app},
we show how our 
fundamental region may be used to study $W$-orbits of finite subsets of $V$, both ordered and unordered. In \S\ref{sec:typeA}, we apply the results of 
\S\ref{sec:app} to show how the conjugacy classes of type $A$-subsystems of an arbitrary root system $\Phi$ may be determined by inspection of the Dynkin diagram of $\Phi$. Finally,   in \S\ref{sec:arbtype}, we indicate, partly without proof,  how the results of \S\ref{sec:app}--\ref{sec:typeA} may be used to study  conjugacy classes of root subsystems of $\Phi$ and thus conjugacy classes
of reflection subgroups of $W$. Related results on conjugacy of subsystems of root systems
may be found in \cite{Osh},  \cite{LT}, \cite{Ro} and  \cite{T}.

\section{Preliminaries}\label{sp}
\subsection{} \label{rootsyst}  Let  $V$ be a real Euclidean space, i.e. a finite dimensional 
real vector space equipped with a symmetric, positive definite bilinear form 
$\mpair{-,-}\colon V\times V\to \real$.  For non-zero  $\a\in V$,  let 
$s_{\a}\colon V\to V$ denote the orthogonal reflection in $\a$; it is the 
$\real$-linear map  defined by $s_{\a}(v)=v-\mpair{v,\ck \a}\a$ where 
$\ck \a:=\frac{2}{\mpair{\a,\a}}\a$. In this paper, by  a   \emph{root system} $\Phi$  in $V$, 
we shall mean  a subset $\Phi$ of $V$ 
satisfying the following conditions (i)--(iii):
\begin{conds}
\item $\Phi$ is a finite subset of $V\sm\set{0}$.
\item If $\a,\b\in \Phi$, then $s_{a}(\b)\in \Phi$.
\item If $\a,c\a\in \Phi$ with $c\in \real$, then $c\in \set{\pm1}$.\end{conds}

The subgroup $W$ of $\mEnd(V)$ generated by $\mset{s_{\a}\mid \a\in \Phi}$ is a finite 
(real) reflection group i.e. a finite  Coxeter group.
A \emph{simple system} $\Pi$ of $\Phi$ is a linearly independent  subset $\Pi\seq \Phi$  such that 
$\Phi=\Phi_{+}\disjun \Phi_{-}$ where 
$\Phi_{+}:=\Phi\cap \real_{\geq 0}\Pi$ 
and $\Phi_{-}=-\Phi_{+}$ (we use the symbol $\disjun$ to indicate a disjoint union). Fix a simple system $\Pi$ (it is well  known that such 
simple systems exist). Then $\Phi_{+}$ is the corresponding positive system of $\Phi$  
and  $S:=\mset{s_{\a}\mid \a
\in \Pi}\seq W$ is called the set of \emph{simple reflections} of $W$. It is well 
known that $(W,S)$ is a Coxeter system.  The subset 
$T:=\mset{s_{\a}\mid \a\in \Phi}=\mset{wsw^{-1}\mid w\in W, s\in S}$  of $W$ 
 is called the set of reflections of $W$.

\subsection{Dual root system} \label{dualrootsys} If $\Phi$ is a root system in 
$V$, then $\ck \Phi:=\mset{\ck\a\mid \a\in \Phi}$ is  also a root system, called the 
\emph{dual root system} of $\Phi$; it has a system of simple roots $\ck \Pi:=\mset {\ck \a
\mid \a\in \Phi}$ with corresponding positive roots 
$\ck \Phi_{+}:=\mset{\ck \a\mid \a\in \Phi_+}$ and 
associated finite  Coxeter system $(W,S)$.
\subsection{Weyl groups} \label{crystrootsys} 
The root system $\Phi$ is said to be crystallographic
if for all $\a,\b\in \Phi$, one has $\mpair{\a,\ck\b}\in \Int$.  In that case, $W$ is a 
finite Weyl group
and one defines the  \emph{root lattice} $Q(\Phi):=\Int\Pi$ and \emph{weight 
lattice}
$P(\Phi):=\mset{\lambda\in V\mid \mpair{\lambda,\ck\Phi}\seq \Int}$.  The corresponding 
 lattices $Q(\ck\Phi)$  and $P(\ck\Phi)$ for $\ck \Phi$ are called the 
  \emph{coroot lattice} and \emph{coweight lattice}  of $\Phi$  respectively.

\subsection{}\label{refsubgp} A subgroup $W'$  of $V$ generated by a subset  of $T$ 
is called a reflection subgroup.
It has a root system $\Phi_{W'}=\mset{\a\in \Phi\mid s_{\a}\in W'}
$. We call $\Phi_{W'}$ a  (root) subsystem of $\Phi$.  A simple 
system (resp., positive system) of a root subsystem of $\Phi$ will 
be called a \emph{simple} (resp., \emph{positive}) \emph{subsystem} of $\Phi$. It is well 
known that $\Phi_{W'}$ has a   
  unique simple system $\Pi_{W'}$ contained in the set of  
  positive roots $\Phi_{+}$ of $W$; the corresponding positive 
  system  is $\Phi_{W',+}:=\Phi_{+}\cap \Phi_{W'}$.

 The reflection subgroups $W_{I}:=\mpair{I}$ generated by 
 subsets $I$ of $S$ are 
 called \emph{standard parabolic subgroups}  and their 
 conjugates are called 
 \emph{parabolic subgroups}.  If $W'=W_{I}$, then
 $\Pi_{W'}=\mset{\a\in \Pi\mid s_{a}\in I}$ and $\Phi_{W',+}=\Phi\cap \real_{\geq 0}\Pi_{W'}$.
 
 \subsection{Fundamental chamber for the $W$-action on $V$}\label{ss:fundcham}
 The subset  $\CC=\CC_{W}:=\mset{v\in V\mid \mpair{v,\Pi}\seq \real_{\geq 0}}$ of $V$ is 
 called the \emph{fundamental chamber} of $W$. In the following Lemma, we collect several
 standard facts concerning this situation, which may be found in \cite{Bour, StYale}.
 
 \begin{lemma}\label{lem:fundcham}\begin{num}\item  Every $W$ orbit on $V$ contains a unique 
 point of $\CC$.
 \item For $v\in \CC$, the stabiliser $\stab_{W}(v):=\mset{w\in W\mid w(v)=v}$  is equal to 
 the standard parabolic subgroup $W_{I}$ where 
 $I:=\mset{s\in S\mid s(v)=v}=\mset{s_{\a}\mid \a\in \Pi, \quad \mpair{\a,v}=0}$.
\item The set of  parabolic subgroups is equal to the set of   the stabilisers of points of $V$. 
It is also equal to the set of  pointwise stabilisers of subsets of $V$.
\item The intersection of any set of parabolic subgroups of $W$ is 
a parabolic subgroup of $W$. In particular, we have 
for $I,J\subseteq S$, $W_I\cap W_J=W_{I\cap J}$.
 \end{num}\end{lemma}
 
 It follows that if $\a\in \Phi$ is any root, then  the $W$-orbit $W\a$ contains 
 a unique element of $\CC_{W}$. A root $\a$ in $ \Phi\cap \CC_{W}$ is said to be 
 a  \emph{dominant root}. If $\Phi$ is irreducible, there are at most two such roots
 (cf. \cite[Lemma 2]{DyLeRef}).  We note that if $V$ were a complex vector space, 
 (c) above would not be true.
 
 \subsection{Determination of $W$-orbit representatives}\label{orbrep}
 An efficient algorithm for computing the unique element in $Wv\cap \CC$ for
 an arbitrary element  $v \in V$ is as follows. For $v\in V$, let 
 $\Phi_{v}:=\mset{\b\in \Phi_+\mid \mpair{\b,v}<0}$ and 
 $n_{v}:=\vert \Phi_{v}\vert$.  If $\Phi_{v}\cap \Pi=\eset$, 
 then $\Phi_{v}=\eset$ and  $v\in \CC$. Otherwise, there exists some $\b\in \Phi_{v}\cap \Pi$; one then has 
 $\Phi_{s_{\beta}(v)} =s_{\beta}(\Phi_{v}\setminus\set{\beta})$ 
 and $n_{s_{\b}(v)}=n_{v}-1$. Continuing thus, one obtains $s_{\b_1}\dots s_{\b_r}v\in\CC$,
 where $r=n_v$.
 
  \subsection{Dynkin diagrams of simple systems and np subsets}\label{diag}  Suppose $\Phi$ is crystallographic.   
  Define a \emph{np subset} $\G$ of $\Phi$ to be a subset
such that for distinct elements $\a,\b\in \G$, one has  $\mpair{\a,\ck \b}\leq 0$. 
Then a subset $\G$ of $\Phi$ is a simple subsystem  if and only if it is a 
linearly independent np subset (see \cite{DyLeRef}).

Define the diagram of  a np subset $\G$ of   $\Phi$  to be 
the Dynkin diagram
(as in \cite[\S 4.7]{K}) of the generalised Cartan matrix $(m_{\a,\b})_{\a,\b\in \G}$ where 
$m_{\a,\b}:=\mpair{\ck \a,\b}$. This is   a graph
with vertex set $\G$ such that if $\a,\b\in \G$ are distinct with
$\vert m_{\a,\b}\vert \geq \vert m_{\b,\a}\vert $, the vertices $\a,\b$ are connected by
$\vert m_{\a,\b}\vert $ lines  and these lines are equipped with an arrow pointing towards
$\a$ if $\vert m_{\a,\b}\vert >1$. Note that the arrow points towards the shorter root if $\a,\b$ 
are of different length, and that if $\a=-\b$, then $\a,\b$ are joined by two lines equipped
with a pair of  arrows in opposite directions. 
It is well known that the  connected components of the diagram of 
$\G$ are certain Dynkin 
diagrams of finite or affine type.
  Further,  the np subset $\G$ of $\Phi$  is a simple subsystem of 
  $\Phi$ if and only if 
  the irreducible components of its  diagram are all  Dynkin diagrams 
  of finite type.
  
 By an \emph{ordered simple subsystem} (resp., \emph{ordered np subset}) of $\Phi$, we mean a tuple $\bb=(\b_{1},\ldots, \b_{n})$ of 
pairwise distinct roots whose underlying set $[\bb]:=\set{\b_{1},\ldots, \b_{n}}$ is a simple subsystem (resp., np subset) of $\Phi$.
 
 \section{Orbits under the diagonal action}\label{s1}
 
 \subsection{Fundamental domain for the diagonal $W$-action}  For  \label{funddom}
each  $n\in \Nat$, let $V^{n}:=V\oplus \ldots\oplus  V$ ($n$ factors)
 with the diagonal $W$-action defined by $w(v_{1},\ldots, v_{n}):=(wv_{1},\ldots, wv_{n})$. 
 (For $n=0$,
 $V^{n}:=\set{0}$ with trivial $W$-action, by convention).
 We identify $V^{n}\times V^{m}$ with $V^{n+m}$ in the natural way. 
 For $\bv:=(v_{1},\ldots,v_{n})\in V^{n}$ and any $m\leq n$ in $\Nat$, 
 define the truncation $\tau_{m}(\bv):=(v_{1},\ldots, v_{m})$ and 
 \begin{equation}W_{\bv,m}:=\stab_{W}(\tau_{m}(\bv))=
 \stab_{W}(v_{1})\cap \ldots \cap \stab_{W}(v_{m}).\end{equation} 
 Thus, $W_{\bv,0}=W$ and each $W_{\bv,m}$ is a parabolic subgroup of $W_{\bv,m-1}$. 
 In particular, $W_{\bv,m}$ is a reflection 
 group acting on $V$, and since its root system has a natural positive 
subsystem given by its intersection with $\Phi_+$, it has a well defined fundamental chamber 
 $\CC_{W_{\bv,m}}\subseteq V$. 
 Note that for each $m$, $W_{\bv,m}\subseteq W_{\bv,m-1}$, whence
 $\CC_{W_{\bv,m}}\supseteq \CC_{W_{\bv,m-1}}$.
 
 Let 
 \begin{equation}\label{eq:deffun}
 \CC^{(n)}_{W}:=\mset{\bv=(v_{1},\ldots, v_{n})\in V^{n}\mid 
\text{\rm for  $m =1,\ldots, n$, } v_{m}\in \CC_{W_{\bv, m-1}}}.
 \end{equation}
 
Let $\bv=(v_{1},\ldots, v_{n})\in V^{n}$ and   $1\leq m\leq n-1$.
 Set $W':=W_{\bv,m}$, $\bv':=(v_{1},\ldots,v_{m})$  
and $\bv'':=(v_{m+1},\ldots,v_{n})$. The definitions immediately imply 
that  
\begin{equation}\label{eq:fundchrec}
\bv\in \CC_{W}^{(n)} \text{ \rm  if and only if $\bv'\in \CC_{W}^{(m)}$ 
and $\bv''\in \CC_{W'}^{(n-m)}$}.
\end{equation}

 The next result will be crucial for the investigation of the action of $W$ on subsets of $V$.
 It identifies a fundamental region for the action of $W$ on $V^n$.
 
 \begin{theorem}\label{thm:fundreg}
 \begin{num}
 \item
 If $\bv=(v_1,\dots v_n)\in \CC^{(n)}_{W}$, 
 then $\Stab_{W}(\bv)=W_{\bv,n}$ is the standard parabolic
 subgroup $W_{I_{\bv}}$ of $W$, where 
 $$
 I_{\bv}=\mset{s_\a\mid \a\in\Pi,\mpair {\a,v_i}=0\text{ for }
 i=1,\dots,n}.
 $$
 \item Every $W$-orbit on $V^{n}$ contains a unique point of $\CC^{(n)}_{W}$. 
 \item $\CC^{(n)}_{W}$ is a convex (and hence connected) subset of $V^n$.
 \end{num}
 \begin{proof}  If $n=0$, (a) is trivial. Let $n>0$. 
 Assume by way of induction that   $W_{\bv,n-1}$ is the standard 
 parabolic subgroup $W_{I_{\tau_{n-1}(\bv)}}$ of $W$. Clearly 
 $\Stab_{W}(\bv)\supseteq W_{I_{\bv}}$, and so it suffices to show that
 any element of $\Stab_{W}(\bv)$ lies in $W_{I_{\bv}}$. Let 
 $w\in \Stab_{W}(\bv)$; then evidently 
 $w\in \Stab_{W}(\tau_{n-1}\bv)=
 W_{I_{\tau_{n-1}\bv}}$. 
 Moreover $v_n\in\CC_{W_{I_{\tau_{n-1}\bv}}}$
 implies that $\Stab_{W_{I_{\tau_{n-1}(\bv)}}}(v_n)$ is the standard parabolic
 subgroup of $W_{I_{\tau_{n-1}(\bv)}}$ generated by 
 $\mset{s_\a\in \Pi_{I_{\tau_{n-1}(\bv)}}\mid\mpair{\a,v_n}=0}$, 
 which is precisely the set $I_{\bv}$. This proves (a).
 
 To prove (b), assume by induction that every $W$-orbit on $V^{n-1}$ contains a 
 unique point of $\CC^{(n-1)}_{W}$.
 By induction, there is an element $w'\in W$ such that $w'(\tau_{n-1}(\bv))\in \CC^{(n-1)}_{W}$. Let
 $\bv':=w'(\bv)=(v_{1}',\ldots ,v_{n}')$. Then $\tau_{n-1}(\bv')\in \CC^{(n-1)}_{W}$. 
 Let $W_{n-1}':=\Stab_{W}(\tau_{n-1}\bv')$; by (a), this is the standard parabolic subgroup
 $W_{I'}$ of $W$, where $I'=\mset{s_\a\mid\a\in \Pi, \mpair{\a,v_i'}=0\text{ for }i=1,\dots,n-1}$.
 Now there is an element $w''\in W_{n-1}'$  such that
 $w''(\bv'_{n})\in \CC_{W'_{n-1}}$. Let $w:=w''w'\in W$.
 Then since $w''$ stabilises $(v_1',\dots,v_{n-1}')$, it is clear that
 $\bv''=w(\bv)=w''(\bv')\in \CC_{W}^{(n)}$. This shows that every element of $V^n$
 may be moved into $\CC^{(n)}_W$ by $W$. We show that the intersection of a $W$-orbit on $V^{n}$ with
 $\CC^{(n)}_W$ is a single point, that is, that no two distinct points of $\CC^{(n)}_W$
 are in the same $W$-orbit.
 
 This will be done by induction on $n$. The case $n=1$ is clear.
 Suppose that $\bv\in\CC^{(n)}_W$ and that $w\in W$ is such that $w(\bv)\in \CC^{(n)}_{W}$.
 Then  $w(\t_{n-1}(\bv))\in \CC^{(n-1)}_{W}$. By induction,
$w\in \Stab_W(\t_{n-1}\bv)$, which by (a) is equal to $W_I$, where $I=\mset{s_\a\mid\a\in\Pi,
\mpair{\a,v_i}=0\text{ for }i=1,\dots,n-1}$.
Since $w(v_{n})\in\CC_{W_I}$ and $v_{n}\in\CC_{W_I}$, it follows that $wv_n=v_n$ and (b) follows.

 To prove (c), let $\bu=(u_1,\dots,u_n)$ and $\bv=(v_1,\dots,v_n)$ be points in $\CC_W^{(n)}$,
 and for $t\in\bbR$, $0\leq t\leq 1$, write $\bx(t)=t\bu+(1-t)\bv$. We wish to show that
 $\bx(t)\in\CC_W^{(n)}$. By induction, we may assume that $tu_m+(1-t)v_m\in \CC_{W_{\bx_t,m-1}}$
 for $m<n$, and we require the corresponding statement for $m=n$. 
 We shall prove first that for $t\neq 0,1$, we have, for all 
 non-negative integers $m$,
 \begin{equation}
 \label{eq:interval}
 W_{\bx(t),m}=W_{\bu,m}\cap W_{\bv,m}.
 \end{equation}
 The inclusion of the right hand side in the left is clear. We prove the reverse inclusion.
For $m=0$, this is trivial.
% For $m=1$ note that if $\a\in\Pi$ and $\a\not\in \Pi_{I_1}\cap \Pi_{I_2}$, 
% where $I_1=\mset{\a\mid\a\in\Pi, \mpair{\a,u_1}=0}$ and 
% $I_2=\mset{\a\mid\a\in\Pi, \mpair{\a,u_1}=0}$, then $\mpair{tu_1+(1-t)v_1,\a}>0$,
% and hence $s_\a\not\in W_{\bx(t)}$. Thus $W_{\bx(t),1}\subseteq W_{\bu,1}\cap W_{\bv,1}$,
% and the reverse inclusion is trivial.
  In general, suppose by way of induction (using (a))
 that $W_{\bx(t),m-1}=W_I$ with $m\geq 1$, where $I=I_1\cap I_2$, and 
 $I_1=\mset{s_\a\mid \a\in\Pi,\mpair{\a,u_i}=0\text{ for }i=1,\dots, m-1}$ and 
 $I_2=\mset{s_\a\mid \a\in\Pi,\mpair{\a,v_i}=0\text{ for }i=1,\dots, m-1}$.
 Then $u_m\in\CC_{W_{I_1}}$, whence $\mpair{u_m,\a}\geq 0$ for $\a\in\Pi_{I_1}$,
 and $v_m\in\CC_{W_{I_2}}$. Now if $\a\in\Pi_I$, then $\mpair{\a,u_m}\geq 0$
 and $\mpair{\a,v_m}\geq 0$. If $\mpair{\a,u_m}\neq 0$ or $\mpair{\a,v_m}\neq 0$,
 then $\mpair{\a,tu_m+(1-t)v_m}>0$, and $s_\a\not\in W_{\bx(t),m}$. This proves the assertion
 \eqref{eq:interval}.
 
 It follows that for $t\neq 0,1$, $\CC_{W_{\bx(t),n-1}}=\mset{v\in V\mid \mpair{v,\alpha}\geq 0 \text{ \rm for all } \a\in I}$, where
 $I=I_1\cap I_2$ is as in the previous paragraph with $m=n$. But $u_n\in\CC_{W_{I_1}}$ and 
 $v_n\in\CC_{W_{I_2}}$, whence for $\a\in\Pi_{I_1}\cap\Pi_{I_2}$, we have 
 $\mpair{\alpha,tu_n+(1-t)v_n}\geq 0$, and (c) is proved.
 \end{proof}
  \end{theorem}

 \subsection{Total orderings on $V$}\label{lexorder} 
 A {\em vector space  total order} $\leq$ of $V$   is a total order of $V$ such that
 the set $\mset{v\in V\mid v>0}$ of positive elements is closed under addition and 
 under multiplication by positive real scalars. One way such orderings arise is as follows.
 Take any (ordered) basis $\set{\a_{1},\ldots, \a_{p}}$ of $V$, and define the total order $\leq$  
 by declaring that $u<v$ if $v-u=\sum_{i=1}^{p}c_{i}\a_{i}$ and there is some $j$ with 
 $c_{i}=0$ for $i<j$ and $c_{j}>0$.  It is well known that in fact, all total orders arise in this way (since $V$ is finite dimensional).
 
    When $V$ contains a root system $\Phi$, 
    fix a vector space total order $\leq$ of $V$ such that 
   every positive root of $W$ is positive in that order i.e.
 $\Phi_{+}\seq \mset{v\in V\mid v>0}$. One such order is the 
 one described above, taking the simple system  $\Pi=\set{\a_{1},\ldots, \a_{l}}$
to be the initial segment of a basis of $V$ as above.
 
 Given such an ordering of $V$, we endow $V^{n}$ with the 
 lexicographic total order $\preceq$ induced by $\leq$ i.e.
 $(v_{1},\ldots ,v_{n})\prec (u_{1},\ldots, u_{n})$  if there is some $j$
 such that $v_{i}=u_{i}$ for $i<j$ and $v_{j}< u_{j}$.
 
 \begin{prop}\label{prop:lexorder} Let $\bv\in  \CC^{(n)}_{W}$. 
 Then $v$ is the maximum element of the orbit $W\bv$ in the total order induced by $\preceq$.
  \end{prop}
\begin{proof} Let $w\in W$ and $\bu:=w\bv\in V^n$. If $w\in W_{\bv,n}$, then $\bu=\bv$.
Otherwise, there is some $j$ with $0\leq j\leq n-1$ such that
$w\in W_{\bv,j}\sm W_{\bv,j+1}$. Write $\bv=(v_{1},\ldots,v_{n})$ and 
$\bu=(u_{1},\ldots, u_{n})$. Then $v_{i}=u_{i}$ for $i\leq j$, 
$v_{j+1}\in \CC_{W_{\bv,j}}$,
and $u_{j+1}=w(v_{j+1})$. Since  $w\in W_{\bv,j}\sm W_{\bv,{j+1}}=
W_{\bv,{j}}\sm \stab_{W}(v_{j+1})$,  it follows that 
$0\neq v_{j+1}- u_{j+1}=v_{j+1}-w(v_{j+1})\in 
\real_{\geq 0}\Pi_{W_{\bv,j}}$. 
Moreover $\Pi_{W_{\bv,j}}\seq \Phi_{+}= \mset{\bv\in \Phi\mid \bv>0}$, 
and so $v_{j+1}>u_{j+1}$. Since $v_{i}=u_{i}$ for $i<j+1$, 
one has $\bu\preceq \bv$ as required. 
\end{proof} 

\begin{corollary}\label{cor:fund}
Given any total order on $V$ such that $0\leq\Phi_+$, let $\preceq$ be the
induced lexicographic  total order on $V^n$. Then the region $\CC^{(n)}_W$ is precisely
the set of maximal elements in the $W$-orbits on $V^n$.
\end{corollary}
\begin{proof}
This is immediate from Theorem \ref{thm:fundreg} and Proposition \ref{prop:lexorder}.
\end{proof}

\subsection{Ordered and unordered sets}\label{ss:ordunord} For a group $H$ acting on the left on a set $U$, we denote the set of $H$-orbits on $U$ by $U/H$. 
We use frequently below the simple fact that if  $H=H_{1}
\times H_{2}$ is a product of two groups, then $H_{1}$ (resp., $H_{2}$) acts 
naturally on $U/H_{2}$ (resp., $U/H_{1})$) and there  are canonical bijections 
\begin{equation}\label{eq:prodorb}
U/H\cong (U/H_{1})/H_{2}\cong (U/H_{2})/H_{1}.
\end{equation}
We record the following elementary observation.

\begin{lemma}\label{lem:tot}
Let $U$ be a totally ordered set, and let $H$ be a finite group acting on $U$.
For any finite subset $A$ of $U$, denote by $\max(A)$ the maximum element of $A$.
Then
\begin{num}
\item The map $Hu\mapsto \max(Hu)$ ($u\in U$) defines a bijection from
the set $U/H$  to a subset of $U$, which we denote by $M_H=M_H(U)$.
\item If $H=H_1\times H_2$ is a direct product of groups $H_1,H_2$, then  for $u\in U$, $\max(Hu)=\max_{h\in H_1}(\max(hH_2u))$.
\item 
%For $u\in U$, write $[u]=H_2u\in U/H_2$. 
The set  $\mset{H_{2}u\mid u\in M_{H_1\times H_2}}$ is a set of distinct representatives
for the $H_1$-orbits on $U/H_2$.
\end{num}
\end{lemma}

\subsection{} Let $n\in \Nat$.  We regard the symmetric group  $\Sym_{n}$ as the group of all permutations of  $\set{1,\ldots, n}$ (acting on the left) and often 
write its elements in cycle notation. Take $U=V^n$, and let $G$ be a subgroup of 
the symmetric group $\Sym_{n}$. Then $G$ has a natural left action on $V^{n}$ by place permutations, defined by 
$\s(v_{1},\ldots, v_{n})=(v_{\s\inv(1)},\ldots, v_{\s\inv(n)})$, which commutes with the 
diagonal $W$-action and induces  a $W\times G$-action on $V^{n}$. Assume chosen, a total order
$\leq$ on $V$, with corresponding total order $\preceq$ on $U=V^n$, as in Corollary \ref{cor:fund}.
Write $H=W\times G$, and recall that for any $H$-action on the ordered space $U$, $M_H(U)$ is
the set of elements of $U$ which are maximal in their $H$-orbit.

\begin{corollary}\label{cor:order}
\begin{num}\item For $\bv\in V^{n}$, $\max(H\bv)=\max\mset{\max( W\s(\bv))\mid \s\in G}$. 
\item  $\mset{G\bv\mid \bv\in M_{W\times G}(V^{n})}$ is a set of orbit representatives 
for $W$ acting  on $V^{n}/{G}$. 
\item We have $M_W(V^n)=\CC^{(n)}_W\supseteq M_{W\times G}(V^{n})$.
 \end{num}\end{corollary} 
\begin{proof} Parts (a) and (b) are immediate from Lemma \ref{lem:tot}.
Part (c) follows from parts (a) and (b) and Corollary \ref{cor:fund}.
\end{proof}

\begin{prop}\label{prop:genclass} 
Let $\wh \CS$ be a subset of $V^{n}$ 
which is stable under the action of the group $H=W\times G$ as above. 
\begin{num}\item The set $\wh \CR:=\wh \CS\cap \CC_{W}^{(n)}$
is in canonical bijection with $\wh \CS/W$.
\item The set of  $W$-orbits on $\wh \CS/G$  is in canonical bijection with the set  $\wh\CS/H$ of $H$-orbits  on $\wh \CS$ and also with the set of  $G$-orbits on $\wh \CS/W$.
\item Define a left ``dot''  $G$-action $(g,\bb)\mapsto g\cdot \bb$  on $\wh\CR$ by transferring the $G$-action on $\wh \CS/W$ via the bijection of {\rm (a)}. This action is determined by either condition $\set{g\cdot \bb}=Wg\bb\cap \wh \CR$ or   $g\cdot \bb=\max(Wg\bb)$   for $\bb\in \wh \CR$ and $g\in G$.
\item  The $G$-orbits in the dot action on $\wh \CR$ are the equivalence classes for the equivalence relation $\simeq$ on $\wh\CR$ defined  by stipulating
(for $\bb,\bb'\in\wh\CR$) that $\bb\simeq\bb'\iff \bb'\in H\bb$. Hence the $W$-orbits on $\wh \CS/G$  are in canonical
bijection with $\wh\CR/{\simeq}$.
\item If $\eta$ is the natural map
$\wh\CS\lr \wh \CS/G$, then the number of $W$-orbits on $\wh\CS/G$ is at most $\vert\eta({\wh \CR\,\,})\vert$.
\end{num}
\end{prop}

\begin{proof} 
Part (a) follows from Theorem \ref{thm:fundreg}.  Part (b)  was already observed more generally in \ref{ss:ordunord}.
We prove (c). Let $\phi\colon  \wh \CR\to \wh \CS/W$ be the canonical bijection of (a).  
By definition, $\phi(\bb)=W\bb$ for $\bb\in \wh \CR$ and $\set{\phi^{-1}(\G)}=\Gamma\cap \wh \CR$ for any $W$-orbit  
$\Gamma\in \wh \CS/W$. Hence the dot action is determined by
$\set{g\cdot \bb}=Wg\bb\cap \wh \CR$ for $g\in G$ and $\bb \in \wh \CR$. But by Corollary \ref{cor:fund}, $Wg\bb\cap \wh \CR=\set{\max(Wg\bb)}$ and (c) follows. 

Now we prove (d). Let $\simeq$ be the equivalence relation on $\wh \CR$ for which the equivalence classes are the $G$-orbits in the dot action. Then by (c),  for $\bb,\bb'\in\wh\CR$ one has $\bb\simeq\bb'\iff\bb'\in \mset{\max(Wg\bb)\mid g\in G}$.  Certainly
$\bb\simeq \bb'$ implies $\bb'\in H\bb$. But if $\bb,\bb'\in\wh\CR$ with $\bb'\in H\bb$, one has $\bb'\in Wg\bb$ for some $g\in G$ and then $\bb'=\max(Wg\bb)$ by Corollary \ref{cor:fund} since $\bb\in \wh \CR$. This proves the first assertion of  (d), and the second then follows from (a).

To see (e), we need only note that the fibres of the restriction of $\eta$ to 
$\wh\CR$ are subsets of the $\simeq$-equivalence classes, whence the 
number of equivalence classes on $\wh\CR$ is at most the number of such fibres.
\end{proof}
\subsection{} Taking $\wh \CS=V^{n}$ and $G=\Sym_{n}$ 
in Proposition \ref{prop:genclass}(c) defines a dot action $(g,\bb)\mapsto g\cdot \bb=g\cdot_{n}\bb$ of  
$\Sym_{n}$ on $\CC^{(n)}_{W}$ for each $n$ (of which all 
possible dot actions as in Proposition \ref{prop:genclass}(c) are 
various restrictions).  We record the following trivial but useful 
relations between these 
$\cdot_{n}$ actions for varying $n$
\begin{prop}\label{prop:dots} Let notation be as above.
Let $\bb=(\b_{1},\ldots, \b_{n})\in \CC_{W}^{(n)}$, $\s\in \Sym_{n}$ and $m\in \Nat$ with $0\leq m\leq n$. Set $\bb':=\t_{m}(\bb)=(\b_{1},\ldots, \b_{m})\in \CC^{(n)}_{W}$,
$W':=W_{\bb'}$ and $\bb'':=(\b_{m+1},\ldots, \b_{n})\in
 \CC^{(n-m)}_{W'}$. Denote the $\cdot$ action of $\Sym_{n-m}$ on $\CC^{(n-m)}_{W'}$ by $\cdot_{n-m}'$.
\begin{num}
\item If $\s\bb\in \CC_{W}^{(n)}$, then $\s\cdot \bb=\s\bb$.
\item  Suppose  $\set{1,\ldots, m}$ is $\s$-stable. Let $\s'
$ be the restriction of $\s$ to a permutation of $\set{1,\ldots, m}$. Then $\tau_{m}(\sigma\cdot_{n}\bb)=\sigma'\cdot_{m}(\tau_{m}\bb)$.
\item Suppose that $\s$ fixes $i$ for $i=1,\ldots, m$.
Let $\s'\in \Sym_{n-m}$ be defined by $\s'(i):=\s(i+m)-m$ for $i=1,\ldots, n-m$.
Then $\s\cdot_{n} \bb=\s\cdot_{n}(\bb',\bb'')=(\bb',\s'\cdot'_{n-m}\bb'')$ i.e. $\s\cdot_{n}\bb=(\b_{1},\ldots, \b_{m},\b_{m+1}',\ldots, \b_{n}')$ where $(\b'_{m+1},\ldots, \b'_{n}):=
\s'\cdot'_{n-m}\bb''$.
\end{num}
\end{prop}
\begin{proof}
This is a straightforward consequence of the definitions.
Details are left to the reader.
\end{proof}
\subsection{Automorphisms} \label{ss:diagaut}
Denote the group of all (linear) isometries of $(V,\mpair{-,-})$ 
which restrict to a permutation of the simple roots $\Pi$  by $D$.
Then $D$ acts diagonally on $V^{n}$ and it is easily seen 
that  $\CC_{W}^{(n)}$ is $D$-stable.  It is well known that if 
$\linspan(\Phi)=V$ and the $W$-invariant 
inner product $\mpair{-,-}$ is chosen suitably
(by rescaling  if necessary its restrictions to the linear spans of the components of $\Phi$), then $D$ 
 identifies with  the group $D'$ of all diagram 
 automorphisms of $\Pi$ (i.e. the 
  group of automorphisms of the 
 Coxeter system $(W,S)$).   It follows that in general,  $\CC_{W}^{(n)}\cap \Phi^{n}$ is invariant under the diagonal 
 action of $D'$ on $\Phi^{n}$ (even if $\linspan(\Phi)\neq V$ or  $\mpair{-,-}$ is not chosen in this way).

 \section{Stratification of $V^{n}$}\label{sec:strat}
 \subsection{Cones} Recall that a subset of a topological space $V$ is \emph{locally closed} if 
 it is open in its closure, or equivalently, if it is the \label{topology} intersection of an 
 open subset and a closed subset of $V$. A  subset of $V$ is said to be constructible if it is a finite union of locally 
 closed subsets  of $V$. By a cone in a real vector space, we mean a subset which is closed under addition and closed  under multiplication by positive scalars. Note that cones may or
 may not be constructible.
\subsection{Facets} \label{ss:facets}
The fundamental chamber $\CC=\CC_{W}$ and certain notions below depend not
only on $W$ and $\Phi$, but  also on the simple system $\Pi$; this dependence will be made 
explicit in notation to be introduced presently.

 For $J\seq I\seq S$ define \begin{equation}\label{eq:facets1}
\CC_{I,J}:=\mset{v\in V\mid
\mpair{\a,v}=0 \text{ \rm for $\a\in \Pi_{J}$ and }
\mpair{\a,v}>0 \text{ \rm for $\a\in \Pi_{I\sm J}$}}
\end{equation}
This is a non-zero locally closed cone  in  $V$. From \cite{Bour}, \begin{equation}\label{eq:facets2}
\CC_{W_{I}}=\bdisjun_{J\seq I } \CC_{I,J},\qquad
\ol{\CC_{I,J}}=\bdisjun_{K: J\seq K\seq I}\CC_{I,K}
\end{equation} 
 For $J\seq I\seq S$, let $W^{J}:=\mset{w\in W\mid w(\Pi_{J})\seq \Phi_{+}}$ and 
 $W^{J}_{I}:=W_{I}\cap W^{J}$. From \cite{Bour},  one has  
 $W_{I}=\bdisjun_{w\in W^{J}_{I}}wW_{J}$ and each element $w\in W^{J}_{I}$ satisfies 
 $l(wx)=l(w)+l(x)$ for all $x\in W_{J}$. 
 
Fix $I\seq S$. The sets $w(\CC_{I,J})$ for $w\in W_{I}$ and $J\seq I$ are called  the 
\emph{facets} (of $W_{I}$ acting  on $V$). They are locally closed cones in $V$, and 
the closure of any facet is a union of facets.  It is well known that any two facets  
either coincide or are disjoint. 
  The setwise stabiliser in $W_I$ of $\CC_{I,J}$ coincides with the  stabiliser in $W_I$ of any 
  point of $\CC_{I,J}$ and is equal to $W_{J}$. It follows that for $I\seq S$,  one has 
  the decomposition \begin{equation}\label{facetdec}
V=\bdisjun_{J\colon I\seq J}\,\,\bdisjun _{w\in W^{J}_{I}}\ w(\CC_{I,J})
\end{equation}
of $V$ as a union of pairwise disjoint facets.

\subsection{Strata} \label{ss:strat}
 The family of subsets $w(\CC_{S,J})$ for $J\sneq S$ 
and $w\in W$, or rather the complex of spherical simplices cut out by
their intersections with the unit sphere in $V$,
 is  known as the \emph{Coxeter complex} of $W$ 
on $V$ (see \cite{Car}, \cite{Hum}, \cite{GL83}); there are also other closely  related meanings 
of the term Coxeter 
complex in the literature,  for instance (cf. \cite{CuLe82}) where the Coxeter complex is 
an (abstract)  simplicial complex or chamber complex with vertices
which are cosets of proper parabolic subgroups. We shall now
define a similar ``complex'' for $W$ acting on $V^{n}$, where  $n\in \Nat_{\geq 1}$,  
and show that it has significant similarities to, and differences from, the Coxeter complex.

  Let \begin{equation}\label{eq:strat1}
 \CI=\CI^{(n)}:=\mset{(I_{0},\ldots, I_{n})\mid S=I_{0}\sreq I_{1}\sreq 
 \ldots \sreq I_{n}}
\end{equation}  denote the set of all weakly decreasing sequences of  $n+1$ 
subsets of $S$ with $S$ as first term.  For $\bI=(I_{0},\ldots, I_{n})\in \CI$, define  
\begin{equation}\label{eq:strat2}
\CX_{\bI}:=\CC_{I_{0},I_{1}}\times \CC_{I_{1},I_{2}}\times \ldots \times
\CC_{I_{n-1},I_{n}}\seq V^{n}.
\end{equation}
The sets $w(\CX_{\bI})$  for $\bI\in \CI^{(n)}$ and $w\in W$ are non-empty 
locally closed cones in $V^{n}$ which  will be called the  \emph{strata} of
$V^{n}$ (for $(W,S)$ or $W$).   Their basic properties are listed in the  Theorem below. 
\begin{theorem}\label{thm:strat}\begin{num} \item  If $\bI\in \CI^{(n)}$  then $\CX_{\bI}\seq \CC^{(n)}$. 
\item  If $\bI\in \CI^{(n)}$, $w\in W$ and $\bv\in  w(\CX_{\bI})$, then 
$\stab_{W}(\bv)=wW_{I_{n}}w^{-1}$.
\item Let $v,w\in W$ and $\bI,\bJ\in \CI^{(n)}$. 
 Then the following conditions are equivalent 
\begin{subconds} \item  $v(\CX_{\bI})\cap w(\CX_{\bJ})\neq \eset$.
\item $v(\CX_{\bI})=w(\CX_{\bJ})$.
\item   $\bI=\bJ$ and $v^{-1}w\in W_{I_{n}}$.
\end{subconds}
\item If $\bI\in \CI^{(n)}$ and  $w\in W$, then   $\stab_{W}(w(\CX_{\bI}))=wW_{I_{n}}w^{-1}$.
\item The sets $\CC^{(n)}$ and $V^{n}$ are the  disjoint unions of the  strata they contain.\item The topological closure of any stratum of $V^{n}$ is a union of  strata.\end{num}
\end{theorem}
\begin{remark} (1)  The fundamental region $\CC^{(n)}_{W}$ is constructible (in fact, it is a finite union of locally closed cones).

 (2) If $n=1$, the fundamental region $\CC_{W}^{(n)}=\CC$ is closed in $V$, but in general,  
 $\CC^{(n)}_{W}$ is a  constructible subset of $V^{n}$ which need not be locally closed;
moreover the closure of a stratum in $\CC_{W}^{(n)}$ may contain strata outside $\CC^{(n)}_{W}$.

(3) If $n=1$, then the facets (here called strata) of $W$ on $V$ are the maximal 
connected subsets of $V$, all points of
which have the same stabiliser. For $n>1$, the stratum containing $\bv$ 
is precisely the connected component containing $\bv$ of the space of all  $\bu\in V^{n}$ such that $\stab_{W}(\tau_{i}(\bu))=\stab_{W}(\tau_{i}(\bv))$ for $i=1,\ldots, n$. Recall that
here $\tau_i$ is the truncation map $V^n\to V^i$
given by $\tau_i(u_1,\dots,u_n)=(u_1,\dots,u_i)$.

 \end{remark}
 
\ssect{An example}\label{ss:A1examp} Before giving its proof, we illustrate Theorem \ref{thm:strat}  and its following remarks in the simplest non-trivial situation.

Let $W=\set{1,s}$, and $S=\set{s}$, regarded as Coxeter group of type $A_{1}$ acting as reflection group on $V:=\real$ with $s$ acting by multiplication by $-1$, with root system $\Phi=\set{\pm 1}$ and unique positive root $\a:=1$.

Then $\CC_{W}=\real_{\geq 0}$, $\CC_{W_{\eset}}=\real$,
$\CC_{S,S}=\set{0}$, $\CC_{S,\eset}=\real_{>0}$ and 
$\CC_{\eset,\eset}=\real$. 
If $n\in \Nat_{\geq 1}$, then $\CC^{(n)}_{W}$ is the set of all 
$(\l_{1},\ldots,\l_{n})\in \real^{n}$ such that if $\l_{j}\neq 0$ and $\l_{i}=0$ for all $i<j$, then $\l_{j}>0$. In other words, it consists of zero and all non-zero vectors in $\real^{n}$ with their first non-zero coordinate positive.  Note that this is the cone of non-negative vectors of a vector space total ordering of $V$.

A typical stratum in $\CC^{(n)}_{W}$ is of the form
$\CX_{i}:=\CX_{S,\ldots, S,\eset,\ldots,\eset}$ for some $0\leq i\leq n$. where there are $n+1$ subscripts on $\CX$, of which the first $n-i+1$ are equal to $S$ and the last $i$ are equal to $\eset$.
One readily checks that $\CX_{0}=s(\CX_{0})=\set{(0,\ldots,0)}$ and that for $i>0$,
\begin{equation*}
\CX_{i}=\set{0}\times \ldots \times \set{0}\times \real_{> 0}\times \real\times \ldots\times  \real=\set{0}^{n-i}\times \real_{>0}\times \real^{i-1}.
\end{equation*}

Thus, there are $2n+1$ distinct strata of  $W$ on $V^{n}$, namely
$\CX_{0}$ and $\CX_{i},s(\CX_{i})$ for $i=1,\ldots, n$.
One readily checks that the closure of a stratum is given by its union
with the strata below it in the following Hasse diagram:
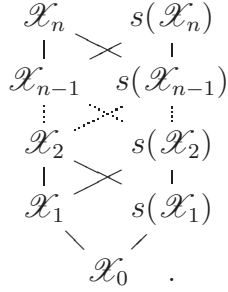
\begin{equation*}
\xymatrix@!0{
{\CX_{n}}&&{s(\CX_{n})}\\
{\CX_{n-1}}\ar@{-}[u]\ar@{-}[urr]&&{s(\CX_{n-1})}\ar@{-}[u]\ar@{-}[ull]\\
{\CX_{2}}\ar@{.}[u]\ar@{.}[urr]&&{s(\CX_{2})}\ar@{.}[u]\ar@{.}[ull]\\
{\CX_{1}}\ar@{-}[u]\ar@{-}[urr]&&{s(\CX_{1})}\ar@{-}[u]\ar@{-}[ull]\\
&{\CX_{0}}\ar@{-}[ur]\ar@{-}[ul]&.
}
\end{equation*}

 \subsection{Proof of Theorem \ref{thm:strat}(a)--(e)}
 Let $\bI\in \CI^{(n)}$ and $\bv\in \CX_{\bI}$.
 Then $v_{i}\in \CC_{I_{i-1},I_{i}}$ for $i=1,\ldots, n$, so $\Stab_{W_{I_{i-1}}}(v_{i})=W_{I_{i}}$. Since $W_{\bv,i}=\stab_{W_{\bv,i-1}}(v_{i})$, it follows by induction that $W_{\bv,i}=W_{I_{i}}$.  By definition, $\bv\in \CC^{(n)}$
 and $\stab_{W}(\bv)=W_{I_{n}}$. This proves (a)   and (b).
 In (c), (iii) implies (ii) by (b) and it is trivial that  (ii) implies (i).  We show 
 that (i) implies (iii).  Suppose that (i) holds: i.e. $v(\CX_{\bI})\cap w(\CX_{\bJ})\neq \eset$.
  That is, for $i=1,\ldots, n$,  $\CC_{I_{i-1},I_{i}}
  \cap v^{-1}w(\CC_{J_{i-1},J_{i}})\neq \eset$.  We have $I_{0}=J_{0}=S$ 
  and $v^{-1}w\in W_{I_{0}}=W$. The properties of facets in \S\ref{ss:facets} imply 
  by induction on $i$ that
 for $i=0,\ldots, n$, $I_{i}=J_{i}$ and $v^{-1}w\in W_{I_{i}}$. This shows
 that (i) implies (iii), and completes the proof of (c). 
 Part (d) follows immediately from (c) and (b).
 
 For (e), we claim that 
 \begin{equation}\label{eq:union}
 \CC^{(n)}=\bigcup_{\bI\in \CI^{(n)}}\CX_{\bI},\qquad  V^{n}=\bigcup_{\bI\in \CI^{(n)}}\bigcup_{w\in W^{I_{n}}}w(\CX_{\bI}).
 \end{equation}
 
 To prove the assertion about $\CC^{(n)}$, note first that the right hand side
 is included in the left. To prove the converse, let $\bv\in \CC^{(n)}$.
 Using Theorem \ref{thm:fundreg},  write  $W_{\bv,i}=W_{I_{i}}$ where $I_{i}\seq S$, 
 for $i=0,\ldots,n$. Then clearly $\bI:=(I_{0},\ldots, I_{n})\in \CI^{(n)}$.  Since $v_{i}\in \CC_{W_{\bv,i-1}}$ and $\Stab_{W_{\bv,i-1}}(v_{i})=W_{\bv,i}$, it follows by 
 induction on $i$ that $v_{i}\in\CC_{I_{i-1},I_{i}}$. Hence $\bv\in \CX_{\bI}$, 
 proving the above formula for $\CC^{(n)}$.
 Since $V^{n }=\cup_{w\in W}w(\CC^{(n)})$, the above assertion concerning $V^{n}$ 
 follows using  (c), which also implies that the facets  in the unions
 \ref{eq:union} are  pairwise distinct and disjoint.

\ssect{Distinguished coset representatives} \label{ss:cosrep}
 The proof of Theorem \ref{thm:strat}(f) (which is given in \ref{stratproof}) involves relationships 
 between closures of  facets with respect to  different parabolic   subgroups
 of $W$. These results  actually apply to arbitrary reflection subgroups,
  and we prove them in that generality (there is no simplification in the 
 formulations or proofs of the results for parabolic subgroups alone). In 
 order to formulate the results, we shall need additional background on 
 reflection subgroups and  more detailed notation 
 which indicates dependence of notions such as  facets,  coset representatives etc  
 on the chosen simple systems for the  reflection subgroups involved. The results needed 
 are simple extensions (given  in \cite{DyRef}) of facts from \cite{Bour} which are
 well known in the case of standard parabolic subgroups. 

Recall that a \emph{simple subsystem} $\G$ of  $\Phi$ is 
defined to be  a simple system $\G$ of some root subsystem of $\Phi$. 
For such a simple subsystem $\G$, let $S_{\G}:=\mset{s_{\a}\mid \a\in \G}$ and $W_{\G}=\mpair{S_{\G}}$.
Then $(W_{\G},S_{\G })$ is a Coxeter system, the length function of which we denote 
as $l_{\G}$. Denote the set of roots of
$W_{\G}$ as $\Phi_{\G}=\mset{\a\in \Phi\mid s_{\a}\in W_{\G}}=W_{\G}\G$ 
and the set of positive roots of $\Phi_{\G}$ with respect to its simple 
system $\G$ as $\Phi_{\G,+}$

 Let $\G$, $\G'$  be    simple subsystems of $\Phi$ 
 such that $W_{\G'}\seq W_{\G}$. Let  
 \begin{equation}\label{eq:cosrep1}
 \tensor*{W}{^{\G'}_{\G}}:=\mset{w\in W_{\G}\mid w(\G')\seq \Phi_{W_{\G},+}}.
\end{equation}
 Evidently one has 
\begin{equation}\label{eq:cosrep1a}
W_{\G}^{u(\G')}=W^{\G'}_{\G}u^{-1}\text{ \rm for all $u\in W_{\G}$}.
\end{equation}
It is known from \cite{DyRef} that under the additional assumption that 
$\Phi_{\G',+}\seq \Phi_{\G,+}$,  
$\tensor*{W}{^{\G'}_{\G}}$ is a set of coset representatives for  
$W_{\G}/W_{\G'}$ and that each element  
 $w\in \tensor*{W}{^{\G'}_{\G}}$ is the unique
element of minimal length in the coset $wW_{\G'}$ of $W_{\G}$
with respect to  the length  function  $l_{\G}$.  Moreover,  
 \begin{equation}  \label{eq:cosrep1b} 
 \tensor*{W}{^{\G'}_{\G}}=\mset{w\in W_{\G}\mid l_{\G}(ws_{\a})>l_{\G}(w)
 \text{ \rm for all $\a\in \G'$}},\quad \text{\rm if $\Phi_{\G',+}\seq \Phi_{\G,+}$} 
 \end{equation}
 
 Now in general if $W_{\G'}\seq W_{\G}$, there is a unique simple system
 $\G''$ for $\Phi_{\G'}$ such that $\G''\seq \Phi_{+}$ and a unique $u\in W_{\G'}$ 
 such that  $\G''=u(\G')$. It follows from \eqref{eq:cosrep1a}
 and the preceding comments that in this more general situation, it is still true that
 $\tensor*{W}{^{\G'}_{\G}}$ is a set of coset representatives for $W_{\G}/W_{\G'}$.
 
  Similarly, define
 \begin{equation}\label{eq:cosrep2}
 \tensor*[^{\G'}]{W}{_{\G}}:=(\tensor*{W}{^{\G'}_{\G}})^{-1}=\mset{w\in W_{\G}\mid
 w^{-1}(\G')\seq \Phi_{W_{\G},+}}.
 \end{equation} 
 This is  a set   of  coset representatives in $W_{\G'}\backslash W_{\G}$,
 each of minimal length in its coset if $\Phi_{\G', +}\seq \Phi_{\G,+}$. 
Further,  
\begin{equation}\label{eq:cosrep2a} 
\tensor*[^{u(\G')}]{W}{_{\G}}=u\,
 \tensor*[^{\G'}]{W}{_{\G}} \text{ \rm for $u\in W_{\G}$}
 \end{equation}  and 
  \begin{equation}  \label{eq:cosrep2b}  \text{\rm if $\Phi_{\G',+}\seq \Phi_{\G,+}$, then } 
  \tensor*[^{\G'}]{W}{_{\G}} =\mset{w\in W_{\G}\mid l_{\G}(s_{\a}w)>l_{\G}(w) 
  \text{ \rm for all $\a\in \G'$}}.
 \end{equation}

%  If $w\in W_{\G}$, we may write $w=w^{\G'}w_{\G'}$ for unique $w_{\G'}\in W_{\G'}$ 
% and $w^{\G'}\in  \tensor*{W}{^{\G'}_{\G}}$. (As is  well known (see \ref{facets}),  
% if $\G'\seq \G$, then $l(w)=l(w^{\G'})+l(w_{\G'})$, but this does not hold in general).  

\ssect{Further notation} \label{ss:funcloslem}  We now expand the notation of 
\S\ref{ss:facets} to include the possibility of non-parabolic reflection subgroups. 
For  any simple subsystem $\G$  of $\Phi$, let  
\begin{equation}\label{eq:fundch}
 \CC_{\G}:=\mset{v\in V\mid \mpair{v,\G}\seq \real_{\geq 0}}=
 \mset{v\in V\mid \mpair{v,\Phi_{\G,+}}\seq \real_{\geq 0}}.
 \end{equation}
 denote the corresponding closed fundamental chamber for $(W_{\G},S_{\G})$ acting on $V$.
 For any $\D\seq \G$, let 
 \begin{equation}\label{eq:facet}
 \CC_{\G,\D}:=\mset{v\in V\mid \mpair{v,\G\sm \D}\seq \real_{>0}, \mpair{v,\D}=0}
 \end{equation} denote the (unique)  facet of $W_{\G}$ on $V$ which is open in 
 $\CC_{\G}\cap \D^{\perp}$.  One easily checks that for $w\in W$, 
\begin{equation}\label{eq:fundchconj}
 \CC_{w(\G)}=w(\CC_{\G}) ,\qquad   \CC_{w(\G),w(\D)}=w(\CC_{\G,\D}).
 \end{equation} 
 The setwise stabiliser of $\CC_{\G,\D}$ in $W_{\G}$ coincides with 
 the stabiliser in $W_{\G}$ of any point of $\CC_{\G,\D}$, which is $W_{\D}$. Moreover,
 \begin{equation}\label{eq:facetclos} \CC_{\G}=\bdisjun_{\D\seq \G}\,\CC_{\G,\D},\qquad 
\ol{ \CC_{\G,\D}}=\bdisjun_{\D'\colon  \D\seq \D'\seq \G}\CC_{\G,\D'}. \end{equation}

 \begin{lemma}\label{lem:funcloslem} Let $\G,\G'$ be simple subsystems of $\Phi$ with
  $\Phi_{\G',+}\seq \Phi_{\G,+}$.  Then \begin{num}
 \item $\CC_{\G}\seq \CC_{\G'}$
 \item $\CC_{W_{\G}}^{(n)}\seq \CC_{W_{\G'}}^{(n)}$ where
 $\CC_{W_{\G}}^{(n)}$ and $\CC_{W_{\G'}}^{(n)}$ are  the 
 fundamental domains we have defined for $(W_{\G},S_{\G})$ and 
 $(W_{\G'},S_{\G'})$ respectively acting on $V^{n}$.\end{num}
 \end{lemma}

\begin{proof}  For (a), observe that
\begin{equation*}
\CC_{\G}=\mset{v\in V\mid \mpair{v,\Phi_{\G,+}}\seq \real_{\geq 0}}\seq 
\mset{v\in V\mid \mpair{v,\Phi_{\G',+}}\seq \real_{\geq 0}}=\CC_{\G'}.
\end{equation*}

 To prove (b), let $\bv\in V^{n}$. Let $W'=W_{\G}$ and $W'':=W_{\G'}$. Recall that   $W_{\bv,i}=\Stab_{W}(\tau_{i}(v))$. Similarly define 
 $W'_{\bv,i}=\Stab_{W'}(\tau_{i}(v))=W'\cap W_{\bv,i}$
and $W''_{\bv,i}=\Stab_{W''}(\tau_{i}(v))=W''\cap W'_{\bv,i}\seq W'_{\bv,i}$; 
they are parabolic subgroups of  $(W_{\G},S_{\G})$ and of $(W_{\G'}, S_{\G'})$ 
respectively with standard  positive systems
\begin{equation*}\Phi_{W''_{\bv,i},+}=\Phi_{W''_{\bv,i}}\cap \Phi_{\G',+}\seq 
\Phi_{W'_{\bv,i}}\cap \Phi_{\G,+} =\Phi_{W'_{\bv,i},+}\end{equation*} 
by the assumption that  $\Phi_{\G',+}\seq \Phi_{\G,+}$. Hence by (a), 
$\CC_{W'_{\bv,i}}\seq  \CC_{W''_{\bv,i}}$ for all $i$. If $\bv\in \CC_{W'}^{(n)}$, then for all $i=1,\ldots,n$,  $v_{i}\in\CC_{W'_{\bv,i-1}}
\seq  \CC_{W''_{\bv,i-1}}$ and so $\bv\in \CC_{W''}^{(n)}$ by definition. 
\end{proof}

\subsection{The main lemma}\label{ss:facclos} 
We now prove the main technical lemma required for the proof 
of Theorem \ref{thm:strat}(f).  
\begin{lemma}\label{lem:facclos}
Let $\G,\G'$ be simple subsystems of $\Phi$ with
  $W_{\G'}\seq W_{\G}$.\begin{num}
 \item $\CC_{\G'}= \bigcup_{w\in\tensor*[^{\G'}]{W}{_{\G}}
 }w(\CC_{\G})$.
 \item If $\D'\seq \G'$, then \begin{equation*}
\ol{\CC_{\G',\D'}}=\bigcup_{w\in\tensor*[^{\G'}]{W}{_{\G}}}\,\,
\bigcup_{\substack{\D\seq \G\\ W_{w(\D)}\sreq W_{\D'}}}w(\CC_{\G,\D}).
\end{equation*}
 \end{num}
 \end{lemma}
\begin{proof} Suppose that (a) holds for  $\G$ and $\G'$. 
Then it also holds for $\G$ and $u(\G')$ for any $u\in W_{\G}$. For   by
\eqref{eq:fundchconj}  and  \eqref{eq:cosrep2a}, one would have 
\begin{equation*}
\CC_{u(\G')}=u(\CC_{\G'})=u\Bigl(\bigcup_{w\in\tensor*[^{\G'}]{W}{_{\G}}
 }w(\CC_{\G})\Bigr)=\bigcup_{w\in\tensor*[^{\G'}]{W}{_{\G}}
 }uw(\CC_{\G})=\bigcup_{w'\in\tensor*[^{u(\G')}]{W}{_{\G}}
 }w'(\CC_{\G}).
\end{equation*} A similar argument shows that if (b) is true for
$\G$,  $\D'$ and  $\G'$, it is true for $\G$,  $u(\D')$ and $u(\G')$ for any
$u\in W_{\G}$.  Since there is $u\in W_{\G}$ with $u(\G')\seq \Phi_{\G,+}$,  
we may and do assume for the proofs of (a)--(b) that $\Phi_{\G',+}\seq \Phi_{\G,+}$.  

To prove (a), note that if $ w\in\tensor*{W}{^{\G'}_{\G}}$, 
then  \eqref{eq:fundchconj} and Lemma \ref{lem:funcloslem}(a) imply that 
$w(\CC_{\G'})=\CC_{w(\G')}\sreq \CC_{\G}$. Hence 
$\cup_{w\in \tensor*[^{\G'}]{W}{_{\G}}}w(\CC_{\G})=\cup_{w\in\tensor*{W}{^{\G'}_{\G}}}w^{-1}(\CC_{\G})\seq \CC_{\G'}$.
To prove the reverse inclusion, let $v\in \CC_{\G'}$.  Write 
$v=w(v')$ where $v'\in \CC_{\G}$ and  $w\in W_{\G}$ is of minimal length  $l_{\G}(w)$. By \eqref{eq:cosrep1b}--\eqref{eq:cosrep2}, it  will suffice to show that 
$l_{\G}(s_{\a}w)\geq l_{\G}(w)$ for all $\a\in \G'$. Suppose first that $\mpair{\a,v}=0$.
Then $v=s_{a}(v)=(s_{\a}w)(v')$ with $s_{\a}w\in W_{\G}$. By choice of $w$, 
$l_{\G}(s_{\a}w)\geq l_{\G}(w)$. On the other  hand,
suppose $\mpair{\a,v}\neq 0$. Since $v\in \CC_{\G'}$, this forces
$0<\mpair{v,\a}=\mpair{w(v'),\a}=\mpair{v',w^{-1}(\a)}$. Since $v'\in \CC_{\G}$ and $w^{-1}(\a)\in \Phi_{W_{\G}}$, it follows that
$w^{-1}(\a)\in \Phi_{W_{\G},+}$ and so $l_{\G}(s_{\a}w)>l_{\G}(w)$ as required.

 Now  we prove (b).  Let $w\in\tensor*[^{\G'}]{W}{_{\G}}$,
  $\D\seq \G$ with $ W_{w(\D)}\sreq W_{\D'}$. Let  
   $v\in \CC_{\G,\D}
 $. Then $v\in \CC_{\G}$, so by (a), $w(v)\in \CC_{\G'}$.
 Since $\mpair{v,\D}=0$, it follows that $\mpair{w(v),w(\D)}=0$ and
  therefore $\mpair{w(v),\D'}=0$ since  $ W_{w(\D)}\sreq W_{\D'}$ .
  Hence $w(v)\in \ol{\CC_{\G',\D'}}$ by \eqref{eq:facetclos}. Thus the right hand side of 
  (b) is included in the left hand side.
  For the reverse inclusion, let  $v\in \ol{\CC_{\G',\D'}}\seq \CC_{\G'}$. By (a),  there 
  exists $w\in\tensor*[^{\G'}]{W}{_{\G}}$ with $v':=w^{-1}(v)\in
   \CC_{\G}$. Thus, $v'\in\CC_{\G,\D}$ for some 
   $\D\seq \G$.
  It remains to prove that $  W_{w^{-1}(\D')}\seq W_{\D}$.
  Let $\a\in \D'\seq \G'$. Since $w\in\tensor*[^{\G'}]{W}{_{\G}}$, it follows that $w^{-1}(\a)\in \Phi_{W_{\G},+}$.  
   Note that 
   \begin{equation*}
  0=\mpair{\a,v}=\mpair{w^{-1}(\a), v'}. 
  \end{equation*} 
  Since $v'\in\CC_{\G,\D}$, one has $s_{w^{-1}(\a)}\in \stab_{W_{\G}}(v')=W_{\D}$. Therefore 
    $  W_{w^{-1}(\D')}\seq W_{\D}$
since the elements  $s_{w^{-1}(\a)}$ for $\a\in \D'$ generate the left hand side. 
This completes the proof of (b).  
\end{proof}
 \begin{remark} One may show that in the union in (b),  $\tensor*[^{\G'}]{W}{_{\G}}$
 may be replaced by $\tensor*[^{\G'}]{W}{_{\G}^{\Delta}}:=\tensor*[^{\G'}]{W}{_{\G}}\cap 
 \tensor*[^{}]{W}{_{\G}^{\Delta}}$, which is a  set of $(W_{\G'}, W_{\D})$ 
 double coset representatives in
 $W_{\G'}\backslash W_{\G}/W_{\D}$, and is the set of all double coset
 representatives which are of   minimal (and minimum) length in their double coset if $\Phi_{\G',+}\seq \Phi_{\G,+}$. (This uses the fact that
 standard facts on  double coset representatives with respect to standard 
 parabolic subgroups on both sides   generalise  to double coset representatives  with respect to an arbitrary reflection subgroup 
 on one side and a standard parabolic subgroup on the other side.)
  After this replacement, the union in (b) is one 
 of pairwise disjoint facets.
 This leads to a similar refinement in \eqref{eq:stratcl}.
  \end{remark}

%%%%%%%%%%%%%%%%%%%%%%%%%%%%%%%%%%%%%%%%%%%%%%%%%%%%%%%%%%%%%%%%%%%%%%%
%%%%%%%%%%%%%%%%%%%%%%%
\ssect{Proof of Theorem \ref{thm:strat}(f)}\label{stratproof}
For a simple subsystem $\G$ of $\Phi$,
and  $n\in \Nat$, let 
\begin{equation}
 \CI_\Gamma=\CI_{\G}^{(n)}:=\mset{\bsy{\G}=(\G_{0},\ldots, \G_{n})\mid 
 \G=\G_{0}\sreq \G _{1}\sreq  \ldots \sreq \G_{n}}
\end{equation}  
denote the set of all weakly decreasing  sequences $\bsy{\G}$ of  $n+1$ 
subsets of $\G$ with $\G$ as first term.  For $\bsy{\G}=(\G_{0},\ldots, \G_{n})\in \CI_\G$, 
define 
\begin{equation}
\CX_{\bsy{\G}}:=\CC_{\G_{0},\G_{1}}\times \CC_{\G_{1},\G_{2}}\times \ldots \times 
\CC_{\G_{n-1},\G_{n}}\seq V^{n}.
\end{equation}
The sets $w(\CX_{\bsy{\G}})$  for $\bsy{\G}\in \CI_{\G}^{(n)}$ and $w\in W_{\G}$ 
are  the  $(W_{\G},S_{\G})$-strata of $V^{n}$.  
If $\bI=(I_{0},\ldots, I_{n})\in \CI^{(n)}$, then $\bsy{\G}:=(\Pi_{I_{0}},\ldots, \Pi_{I_{n}})\in\CI_{\Pi}^{(n)}$ 
and $\CX_{\bI}=\CX_{\bsy{\G}}$.
It is easy to see that the collection of strata of $V^{n}$ with respect to $(W_{\G},S_{\G})$ 
depends only on the reflection subgroup $W_{\G}$ and not on the chosen simple system 
$\G$. 

There is a left action of $W$ on $\cup_{\G}\CI_{\G}^{(n)}$, where the union is over 
simple subsystems $\G$ of $\Phi$,  defined as follows: 
for $w\in W$ and $\bsy{\G}=(\G_{0},\ldots, \G_{n})\in\CI_{\G}^{(n)}$, one has 
$w(\bsy{\G}):=(w(\G_{0}), \ldots, w(\G_{n}))\in \CI_{w(\G)}^{(n)}$ .  
By \eqref{eq:fundchconj}, this action satisfies
\begin{equation}\label{eq:stratact}
w(\CX_{\bsy{\G}})=\CX_{w(\bsy{\G})}, \text{ \rm for all $w\in W$}.
\end{equation}     The setwise stabiliser of $\CX_{\bsy{\G}}$ in $W_{\G}$ is 
equal to the stabiliser in $W_{\G}$ of any  point of $\CX_{\bsy{\G}}$, which is $W_{\G_{n}}$.

 Theorem \ref{thm:strat}(f)  follows from the special case $\G=\L=\Pi$ 
 and  $W=W_{\G}=W_{\L}$ of the following (superficially stronger but actually equivalent) result, which has a simpler inductive proof because of the greater generality of its hypotheses.

 \begin{theorem}\label{thm:stratproof} Let $\G$ and $\L$ be simple subsystems of $\Phi$ 
 with $W_{\G}\seq W_{\L}$.   Then for all $n\in \Nat_{\geq 1}$,  the 
 closure of any $(W_{\G},S_{\G})$-stratum
 $F'$ of $V^{n}$ is a union of  $(W_{\L},S_{\L})$-strata $F$  of $V^{n}$.
  \end{theorem}
 \begin{proof}  A typical stratum $F'$ of $V^{n}$ for $W_{\G}$ is, by the definitions 
 and  \eqref{eq:stratact}, of the form $F'=u(\CX_{\bsy{\G}} )=\CX_{u(\bsy{\G})}$ for 
 some $u\in W_{\L}$ and $\bsy{\G}\in \CI^{(n)}_{\G}$. Replacing $\G$ by $u(\G)$, we 
 may assume without loss of generality  that $u=1$. 
   It will therefore  suffice to establish   the following formula:
 for $\bsy{\G}=(\G=\G_{0},\ldots, \G_{n})\in \CI^{(n)}_{\G}$: 
 \begin{equation}\label{eq:stratcl}
 \ol{\CX_{\bsy{\G}}}=\bigcup_{\substack{{\bw\in W^{n}}\\
 {\bsy{\L}\in \CI^{(n)}_{\L}} \\
 P(\bw,\bsy{\L})}}w_{n}\cdots w_{1}(\CX_{\bsy{\L}}).
 \end{equation}
 The union in \eqref{eq:stratcl} is taken over certain sequences 
 $\bw=(w_{1},\ldots, w_{n})\in W^{n}$ and 
 ${\bsy{\L}=(\L=\L_{0},\ldots, \L_{n})\in \CI^{(n)}_{\L}}$ satisfying the  conditions 
 $P(\bw,\bsy{\L})$(i)--(ii) below
 \begin{conds}
 \item For $i=1,\ldots, n$, $w_{i}\in \tensor*[^{\G_{i-1}}]{W}{_{w_{i-1}\cdots w_{1}(\L_{i-1})}}$.
%  (in particular, $\G_{i-1}\seq \Phi_{w_{i}\cdots w_{1}(\L_{i-1}),+}$ i.e. $\G_{i-1}\seq w_{i}\cdots w_{1}( \Phi_{\L_{i-1},+}))$. 
 \item For  $i=1,\ldots, n$, $W_{w_{i}\cdots w_{1}(\L_{i})}\sreq W_{\G_{i}}$.
 \end{conds}
 For fixed $P(\bw,\bsy{\L})$, we denote these conditions as (i)--(ii).
 Note that the condition (ii) with $i=n$ implies that for $F':=\CX_{\bsy{\G}}$
 and $F:=w_{n}\cdots w_{1}(\CX_{\bsy{\L}})$, one has 
 $\stab_{W_{\L}}(F)=W_{w_{n}\cdots w_{1}(\L_{n})}\sreq W_{\G_{n}}=  \stab_{W_{\G}}(F')$.

 We shall prove \eqref{eq:stratcl} (and that the conditions (i)--(ii) make sense)
 by induction on $n$. If $n=1$, then \eqref{eq:stratcl} reduces to Lemma \ref{lem:facclos}(b).  
 Now assume by way of induction that \eqref{eq:stratcl} holds  and consider
 $\G'=(\G,\G_{n+1})=(\G_{0}, \ldots, \G_{n},\G_{n+1})\in \CI^{(n+1)}_{\G}$.
 Then 
  \begin{equation*}
 \ol{\CC_{\bsy{\G'}}}=\ol{\CX_{\bsy{\G}}}\times\ol{\CC_{\G_{n},\G_{n+1}}} =
 \bigcup_{\substack{{\bw\in W^{n}}\\
 {\bsy{\L}\in \CI^{(n)}_{\L}} \\
 P(\bw,\bsy{\L})}}w_{n}\cdots w_{1}(\CX_{\bsy{\L}})\times\ol{\CC_{\G_{n},\G_{n+1}}}
 \end{equation*}
Fix $\bw\in W^{n}$ and $\bsy{\L}\in \CI_{\L}^{(n)}$ satisfying $P(\bw,\bsy{\L})$ 
and write $w:=w_{n}\cdots w_{1}$. Then  since $W_{w(\L_{n})}\sreq  W_{\G_{n}}$, 
Lemma \ref{lem:facclos}(b) gives
 \begin{equation*}
w_{n}\cdots w_{1}(\CX_{\bsy{\L}})\times\ol{\CC_{\G_{n},\G_{n+1}}}
=\CX_{w(\bsy{\L})}\times\ol{\CC_{\G_{n},\G_{n+1}}}
=\CX_{w(\bsy{\L})}\times \bigcup_{w',\Sigma }w'(\CC_{w(\L_{n}),\Sigma})
 \end{equation*}
 where the union is over all $w'\in \tensor*[^{\G_{n}}]{W}{_{w(\L_{n})}}$ and 
 $\Sigma\seq w(\L_{n})$ with      
 $W_{w'(\Sigma)}\sreq W_{\G_{n+1}}$. Writing
 $w'=w_{n+1}$ and $\Sigma=w(\L_{n+1})$ gives  \begin{equation*} \CX_{w(\bsy{\L})}\times 
 \bigcup_{w',\Sigma }w'(\CC_{w(\L_{n}),\Sigma})=
\bigcup_{w_{n+1},\L_{n+1}}\CX_{w(\bsy{\L})}\times w_{n+1}(\CC_{w(\L_{n}),w(\L_{n+1})})
 \end{equation*} where the union on the right is  taken over all 
 $w_{n+1}\in  \tensor*[^{\G_{n}}]{W}{_{w(\L_{n})}}$ and $ \L_{n+1}\seq\L_{n}$ with $W_{w_{n+1}w(\L_{n+1})}\sreq W_{\G_{n+1}
}$.  Since $w_{n+1}\in W_{w(\L_{n})}=\stab_{W_{\L}}(\CX_{w(\bsy{\L})}) $,  it follows  
using \eqref{eq:stratact} and \eqref{eq:fundchconj}   that
 \begin{equation*}\begin{split}
\CX_{w(\bsy{\L})}\times w_{n+1}(\CC_{w(\L_{n}),w(\L_{n+1})})&=w_{n+1}(\CX_{w(\bsy{\L})})\times
w_{n+1}(\CC_{w(\L_{n}),w(\L_{n+1})})\\ &=
w_{n+1}w(\CX_{\bsy{\L}}\times \CC_{\L_{n},\L_{n+1}})=
w_{n+1}w(\CX_{\bsy{\L'}} )  \end{split} \end{equation*}
 where $\bsy{\L'}:=(\L_{0},\ldots, \L_{n+1})\in
\CI^{(n+1)}_{\L}$. Observe that the conditions imposed on $w_{n+1},\L_{n+1}$ in the last 
union are precisely those which ensure  that
$w_{n+1}$, $\L_{n+1}$ satisfy the conditions on $w_{i},\L_{i}$ in (i), (ii) with $i=n+1$,
and that $\L':=(\L_{0},\ldots, \L_{n+1})\in \CI^{(n+1)}_{\L}$. 
Combining the last four displayed equations with this observation establishes the validity
of \eqref{eq:stratcl} with $n$ replaced by $n+1$. This completes the inductive step, 
and the proof of
Theorem \ref{thm:stratproof}.
\end{proof}
The proof of Theorem \ref{thm:strat} is now complete.

 \ssect{Geometry} \label{ss:facinc}  Define the dimension $\dim(C)$ of any non-empty cone $C$
 in a finite-dimensional real  vector space by $\dim(C)=\dim_\real(\real C)$, where
  $\real C$ is the subspace spanned by $C$.  
  It is well-known that the dimension of any cone contained in $\ol{C}\sm C$ is 
  strictly smaller than the dimension of $C$.
 
 \begin{corollary}\label{cor:facinc} Maintain the hypotheses of 
 \text{\rm  Theorem \ref{thm:stratproof}}, so that in particular $W_{\G}\seq W_{\L}$. 
 \begin{num}
 \item The closure of any $(W_{\G},S_{\G})$-facet $F$ of $V^{n}$ is  the disjoint 
 union of  $F$ and of $(W_{\G},S_{\G})$-strata of $V^{n}$ whose dimension
 is strictly less than  $\dim(F)$.
\item Any $(W_{\G},S_{\G})$-stratum $F'$ of $V^{n}$ is a union of  
certain $(W_{\L},S_{\L})$-strata $F$ of $V^{n}$.   
 \end{num}  \end{corollary}
\begin{proof}
 Write $d:=\dim V$.  Let $\G$ be a simple subsystem 
  and $\D\seq \G$. 
   Then  one has $\dim (\CC_{\G,\D}):=\dim (\real \CC_{\G,\D})=d-\vert \D\vert$.  
   In fact, there is an isomorphism of vector spaces
  $V\xrightarrow{\cong} \real^{d}$ which induces a homeomorphism 
  $\CC_{\G,\D}\xrightarrow{\cong} \real^{d-\vert \G\vert}\times \real_{>0}^{\vert \G\sm \D\vert}$ 
  and
 $\ol{\CC_{\G,\D}}\xrightarrow{\cong} \real^{d-\vert \G\vert}
 \times \real_{\geq 0}^{\vert \G\sm \D\vert}$  where $\real_{\geq 0}^{0}=\real^{0}:=\set{0}$
 and $\real^{m}$ is identified with $\real^{m-i}\times \real^{i}$ for $0\leq i\leq m$.
 
 It follows from  the above and the definitions that for $\bsy{\G}\in \CI^{(n)}_{\G}$, one has
 $\dim(\CX_{\bsy{\G}})=\sum_{i=1}^{n}\dim \CC_{\G_{i-1},\G_{i}}=nd-
 \sum_{i=1}^{n}\vert \G_{i}\vert$. Also, there is  an isomorphism
 $V^{n}\xrightarrow{\cong} \real^{nd}$  inducing  homeomorphisms
 \begin{equation}\label{eq:standfac}\CX_{\bsy{\G}}\cong
  \real^{nd-\sum_{i=1}^{n} \vert \G_{i-1}\vert }\times \real_{>0}^
  {\vert \G_{0}-\G_{n}\vert},\qquad 
\ol{\CX_{\bsy{\G}}}\cong \real^{nd-\sum_{i=1}^{n} \vert \G_{i-1}\vert }
\times \real_{\geq 0}^{\vert \G_{0}-\G_{n}\vert}
 \end{equation} 

 Note that any cone in $\ol{\CX_{\bsy{\G}}}\sm \CX_{\bsy{\G}}$ has 
 dimension strictly smaller than that of $\CX_{\bsy{\G}}$
 (either by the general fact mentioned above or by a direct check in this case). 
 Part (a) follows readily from this fact and  the special case of Theorem 
 \ref{thm:stratproof} in which  $\L=\G$. Given (a), (b) follows from Theorem
 \ref{thm:stratproof} by induction on $\dim(F')$ as follows: (b)
 holds vacuously for strata $F'$ of negative dimension (since there are none).
 In general,  $\ol{F'}$ is a union of $(W_{\L},S_{\L})$-strata by Theorem \ref{thm:stratproof},  $\ol{F'}\sm {F'}$ is a union of $(W_{\G},S_{\G})$-strata of dimension less 
 than $\dim(F')$ by (a) and hence is a union of $(W_{\L},S_{\L})$-strata by 
 induction, and therefore $F'= \ol{F'}\sm(\ol{F'}\sm {F'})$ is a union of 
 $(W_{\L},S_{\L})$-strata as asserted. 
 \end{proof}
  \ssect{Character formulae}\label{ss:chars} 
  We finish this section with a character-theoretic application of Theorem \ref{thm:strat}. Assume for simplicity that $\real \Pi=V$.
  The Coxeter complex $\{w(\CC_{S,J})$ for $J\sneq S$, $w\in W\}$ provides a subdivision of
 the unit sphere $\bbS(V)$ in $V\cong\real^\ell$ into spherical simplices. Applying the Hopf trace formula to the resulting chain complex, and recalling that 
 $\dim(w(\CC_{S,J})\cap\bbS(V))=\ell-1-|J|$, one obtains the familiar character formula (due to Solomon)
 
 \begin{equation}\label{eq:sol}
 \det{_V}(w)=\sum_{J\subseteq S}(-1)^{|J|}\Ind_{W_J}^W(1)(w)\text{    for  } w\in W.
 \end{equation}
 
 It follows from  Theorem \ref{thm:strat} and \eqref{eq:standfac} that the intersections of the strata $w(CX_{\bI})$, where $w\in W$ and $\bI\neq (S,S,\dots,S)$, with the unit sphere
 $\bbS(V^n)$ give a subdivision of $\bbS(V^n)$ into spherical cells (homeomorphic to open balls), and one may again apply the Hopf trace formula to the resulting chain 
 complex. A straightforward computation, using the fact that for any $u\in W$,
$\dim(u\CX_{\bI}\cap \bbS(V^n))=n\ell-(|I_1|+|I_2|+\dots+|I_n|)-1$,
  then shows that, given the formula \eqref{eq:sol}, we have the following formula for $w\in W$.
 
  \begin{equation}\label{eq:sol2}
 \det{_{V^n}}(w)=\det{_V}(w)^n=\sum_{\bI=(I_0,I_1,\dots,I_n)\in \CI^{(n)}}(-1)^{(|I_1|+|I_2|+\dots+|I_n|)}\Ind_{W_{I_n}}^W(1)(w).
 \end{equation}
 
 \begin{rem}
 It is an easy exercise to show that for fixed $I_n\subseteq S$, we have 
 \[
 \sum_{\bI=(I_0,I_1,\dots,I_n)\in \CI^{(n)}}(-1)^{(|I_1|+|I_2|+\dots+|I_n|)}=
 \begin{cases}
 (-1)^{n|S|}\text{ if }I_n=S\\
 0\text{ if } I_n\subsetneq S\text{ and }n\text{ is even}\\
 (-1)^{|I_n|}\text{ if } I_n\subsetneq S\text{ and }n\text{ is odd}.\\
 \end{cases}
 \]
 
 It follows that when $n$ is even, \eqref{eq:sol2} is amounts to the statement that $\det^n=1_W$, 
 while if $n$ is odd, the right side of \eqref{eq:sol2} reduces to the right side of \eqref{eq:sol},
 and therefore amounts to the statement that for $n$ odd, $\det^n=\det$.
 \end{rem}

\section{Topological and combinatorial properties of the stratification.}\label{s:minmax}
Maintain the assumptions of \S\ref{ss:strat}. Write $d:=\dim(V)$.
Let $\CF=\CF^{(n)}:=\mset{w(\CX_{\bI})\mid \bI\in \CI^{(n)}_{W}, w\in W}$ denote the 
set of all $W$-strata of $V^{n}$, partially ordered by inclusion of closures 
of strata; i.e. for $F,G\in \CF^{(n)}$, we say that $F\leq G$ if
$\ol{F}\seq \ol{G}$. The fact that  this defines a partial order follows from
Corollary \ref{cor:facinc}(a).  Note that $W$ acts naturally on $\CF^{(n)}$ as a
group of order preserving automorphisms. 

Let  $\ol{\CF}^{(n)}:=\mset{(w,{\bI})\mid  \bI\in \CI^{(n)}_{W}, w\in W^{I_{n}}}$. The map $(w,\bI)\mapsto w(\CX_{\bI})\colon \ol{\CF}^{(n)}\to
\CF^{(n)}$ is a bijection, by  Theorem \ref{thm:strat}(b). We use this
bijection to transfer the  partial order and  the $W$-action on 
$\CF^{(n)}$ to a partial order and  the $W$-action on $\ol{\CF}^{(n)}$.
Using   \eqref{eq:stratcl}, one sees that this partial order and 
$W$-action $\ol{\CF}^{(n)}$ have a purely combinatorial description in 
terms of the Coxeter system $(W,S)$; this is in analogy with the description
of the Coxeter complex in terms of cosets of standard parabolic subgroups.
In particular, 
$\ol{\CF}^{(n)}$ would be unchanged (as poset with $W$-action) if  it  
had been defined using the  diagonal $W$-action on $(\real \Phi)^{n}$ instead 
of that on $V^{n}$.

\begin{lemma}\label{lem:minmax} Let $\bI_{-}:= (S,S, \ldots, S)\in  \CI^{(n)}_{W}$ and 
$\bI_{+}:= (S,\eset,\ldots, \eset )\in  \CI^{(n)}_{W}$. \begin{num}
\item The poset $\CF^{(n)}$ has a minimum element  $\hat{0}:=\CX_{\bI_{-}}$. 
\item The elements $w(\CX_{\bI_{+}})$ for $w\in W$ are the distinct maximal 
elements of $\CF^{(n)}$.\end{num}
\end{lemma}
\begin{proof} Note that by Theorem \ref{thm:strat}, $\hat{0}$ is fixed by 
the $W$-action. To show that $\hat{0}$ is the minimum element of $\CF$, it  
therefore suffices to show that if
 $\bI\in \CI^{(n)}_{W}$, one has $\hat{0}\seq \ol{\CX_{\bI}}$.
 This is clear since  by \eqref{eq:strat2} and  \eqref{eq:facets2} 
 \begin{equation*}
 \ol{\CX_{\bI}}=\ol{\CC_{I_{0},I_{1}}}\times \ldots \times \ol{\CC_{I_{n-1},I_{n}}}
 \sreq \ol{\CC_{I_{0},I_{0}}}\times \ldots \times \ol{\CC_{I_{n-1},I_{n-1}}}\sreq\ol{\CC_{S,S}}
 \times \ldots \times \ol{\CC_{S,S}}=\ol{\CX_{\bI_{-}}}
 \end{equation*} 
 since $\CC_{I,I}$ is the set of all points of $V$ fixed by $W_{I}$.  
 This proves (a).
 
 A similar argument  using $\CC_{\eset,\eset}=V$ shows that $\ol{\CX_{\bI_{+}}}\sreq \CX_{\bI}$
 for all  $\bI\in \CI^{(n)}_{W}$.  Since by definition $\CF=\mset{w(\CX_{\bI})\mid w\in W, \bI\in \CI^{(n)}_{W}}$, this implies that any 
 maximal stratum in  $\CF$ is of the form   $w(\CX_{\bI_{+}})$ for some $w\in W$.
 But $W$ acts simply transitively on the set of these strata and there is at least 
 one maximal element, so (b) follows.
\end{proof}

%%%%%%%%%%%%%%%%%%%%%%%%%%%%%%%%%%%%%%%%%%%%%%%%%%%%%%%%%%%%%%%%%%%%%%%%%%%%%%%%%%%%%%

\ssect{Topology} \label{ss:celldim} In this subsection, we discuss basic topological 
facts about the stratification of $V^{n}$.
For $m\in \Nat$, let $\bbB^{m}$ denote the standard $m$-ball in 
$\real^{m}$ and  $\bbS^{m-1}$ its boundary,  the standard  $m$-sphere 
(with $\bbS^{-1}:=\eset$).

Let $U$ be a finite-dimensional real vector space. A \emph{ray} in $U$ is a 
subset of $U$ of the form $\real_{>0}v$ for some non-zero $v\in U$.
Let $\CR=\CR_{U}:=\mset{\real_{>0}v\mid v\in U, v\neq 0}$ denote the set 
of all rays in $U$. Topologise $\CR$  as follows. 
  Let $K$ be a convex body (i.e. a compact 
convex  set with non-empty interior)  in $U$ with $0$ in its interior,
so that $K$ contains a small ball with centre $0$.  
Let $\partial(K)$ denote the boundary of $K$ i.e. the set of all 
non-interior points of $K$. The map $v\mapsto \real_{>0}v\colon \partial(K)\to \CR$ is 
a bijection and we topologise $\CR$ so this map is a homeomorphism.
A compactness argument shows the resulting topology is independent 
of  choice of $K$. Taking $K$ as the unit sphere in $U$ (with respect to 
some Euclidean space  structure on $U$) gives 
$\CR\cong \bbS^{\dim (U)-1}$. 

There is a map 
$C\mapsto [C]:=\mset{\real_{>0}v\mid v\in C\sm\set{0}}$ taking convex 
cones $C$ in $U$ to subsets of  $\CR$.
Clearly, $[C]=[C']$ if and only if $C'\cup\set{0}=C\cup\set{0}$.
This map  satisfies $[\ol{C}]=\ol{[C]}$ and $[\inter({C})]=\inter([C])$ where 
$\ol{X}$ and $\inter(X)$ denote respectively the closure and interior of a subspace 
$X$ of the  ambient topological space ($U$ or $\CR$).

We apply the above considerations with $U=V^{n}$. Recall that here $\dim(V)=d$.
\begin{lemma}\label{lem:celldim}
\begin{num}
\item  $[\hat 0]=[\CX_{\bI_{-}}]\cong \bbS^{n(d-\vert S\vert)-1}$.
\item  If $F=w(\CX_{\bI})\in \CF^{(n)}\sm\set{\hat 0}$,  then  $[\ol{F}]\cong
\bbB^{N}$ where $N=nd-1-\sum_{i=1}^{n}\vert I_{i}\vert$ and $[F]\cong
\bbB^{N}\sm\partial(\bbB^{N}):=\bbB^{N}\sm \bbS^{N-1}$. 
 \end{num}
\end{lemma}
\begin{proof} Note that for $\bI\in \CI^{(n)}_{W}$, one has
$ {\CX_{\bI}}={\CX_{\bI_{-}}}$ if and only if $I_{0}=I_{n}$.
Using  \eqref{eq:standfac} and the independence of the topology on $\CR_{V^{n}}$ 
from the choice of compact body $K$ in its definition, it suffices to verify that 
for  $m\leq n\leq M$ in $\Nat$ with $n\geq 1$, the following equations hold 
 in $\real^{M}=\real^{m}\times \real^{n-m}\times \real^{M-n}$: \begin{equation*}
(\real^{m}\times \real_{\geq 0}^{n-m})\cap \bbS^{M-1}\cong 
\begin{cases}\bbS^{n-1},&\text{if $m=n$}\\
\bbB^{n-1},&\text{if $m<n$}
\end{cases}
\end{equation*}
and 
$
(\real^{m}\times \real_{> 0}^{n-m})\cap \bbS^{M-1}\cong 
\partial(\bbB^{n-1}) \text{ \rm if $m<n$}
$.
 Details are omitted.\end{proof}

 \subsection{Regular cell decompositions}\label{ss:regcell} 
 We shall use below  notions of regular cell complexes
 and their face posets, and  shellability of posets and complexes. A  
convenient reference for the definitions and  facts  required   is \cite[4.7]{BjMat}. 
\begin{prop}\label{prop:regcell}
\begin{num}
\item
Suppose that $V=\real \Pi$. Then $\mset{[F]\mid F\in \CF\sm\set{\hat 0}}$ is 
(the set of open cells of) a regular cell decomposition of 
$\CR_{V^{n}}\cong \bbS^{nd-1}$ where $d:=\dim V=\vert S\vert$.
\item The poset $\ol{\CF}^{(n)}\sm\set{\hat 0}$ is  the face poset of 
a regular cell decomposition of  $\bbS^{n\vert S\vert -1}$.
\end{num}
\end{prop}
\begin{proof} First we prove (a). Let $\O:=\CF\sm\set{\hat 0}$, regarded as poset.  
For $F\in \O$, call $[F]\seq \CR$ an open cell and its closure 
$\ol{[F]}=[\ol{F}]$ a closed cell.  By Corollary \ref{cor:facinc} and Lemma \ref{lem:celldim}
(and the discussion preceding its statement), each closed cell $\ol{[F]}=[\ol{F}]$ 
is a ball in $\CR$ with $[F]$ as its interior and   with boundary 
\begin{equation*}
\ol{[F]}\sm [F]=\cup_{G\in \O, G< F}[G]=\cup_{G\in \O, G< F}[\ol{G}]
\end{equation*} 
equal to a union of closed cells. Since $\CR$ is Hausdorff, (a) 
follows by the definition in \cite{BjMat}.  By the discussion immediately 
preceding Lemma \ref{lem:minmax},   $\ol{\CF}^{(n)}\sm\set{\hat 0}$ is the 
face poset of the  regular cell complex in (a), and (b) follows.\end{proof}

\subsection{} Proposition \ref{prop:regcell} has a number   of combinatorial and 
topological consequences listed in \cite[4.7]{BjMat}.
In particular, the finite  poset   $\wh{\CF}^{(n)}:=\ol{\CF}^{(n)}\cup\set{\hat 1}$ 
obtained by formally  adjoining  a maximum element $\hat 1$ to  $\ol{\CF}^{(n)}$ is 
graded (i.e. has a minimum  element  and a maximum element, and all its maximal chains 
(totally ordered subsets) have the same cardinality).
Note that  $\wh{\CF}^{(n)}$ is  called the \emph{extended face poset} of the regular 
cell complex in Proposition \ref{prop:regcell}(a).

We conclude with the remark that significant parts of the above 
results (though not the regular cell subdivisions of spheres in 
Proposition \ref{prop:regcell}, for example)
extend mutatis mutandis to the  diagonal action of  infinite Coxeter groups 
on powers $U^{n}$ of their Tits cone $U$. 

\section{Applications to conjugacy of sets of roots and vectors}\label{sec:app}
\subsection{} Let $\real^{m\times n}$ denote the set of real $m\times n$ matrices and $A\mapsto A^{t}$ denote the matrix transpose. Identify $\real^{n}=\real^{n\times 1}$.

We shall be addressing the classification of certain tuples $(v_1,\dots,v_n)\in V^n$ up to order, under the action of $W$. Evidently
the $W$-action leaves invariant the set of inner products $\langle v_i,v_j\rangle$, $1\leq i,j\leq n$. Arranging these inner products in a matrix 
motivates the following definition.
 \begin{definition}\label{def:genus} 
 A \emph{genus} of \emph{rank} $n$ is a $n\times n$ matrix $\s=(a_{i,j})_{i,j=1}^{n}$ of real numbers. 
 The symmetric group $\Sym_{n}$ acts on the set $\CG^{(n)}=\real^{n\times n}$ of  genera of rank $n$ in such a way that, for $\rho \in \Sym_{n}$ and $\s\in \CG^{(n)}$,  one has  
 $\rho\sigma:=\s'=(a'_{i,j})_{i,j=1}^{n}$ where $a'_{i,j}=a_{\rho^{-1}(i),\rho^{-1}(j)}$.
  The \emph{automorphism group} $G_{\s}$ of the genus $\s$ is the stabiliser of $\sigma$ in this action.
  Two genera of the same rank $n$ are said to be of the same \emph{type} if they are in the same  $\Sym_{n}$ orbit on $\CG^{(n)}$. More formally, we define a \emph{type} of genera of rank $n$ to be a $\Sym_{n}$ orbit 
  on $\CG^{(n)}$ and write $I(\s):=\Sym_{n}\s$ (the $\Sym_{n}$-orbit of $\s$) for the type of the genus $\sigma$.
   \end{definition}
    For example, the Cartan matrices  of fixed rank $n$
    can be regarded as genera, and two of the them have the 
    same type (in the above sense) if and only if they have the 
    same type (of form  $A_{1}^{n_{A_{1}}}\times \cdots \times G_{2}^{n_{G_{2}}}$) in the usual sense. Similar remarks apply to generalised Cartan matrices.
      Thus, we may regard types, in the usual sense,  of 
      (generalised) Cartan matrices as (special) types of 
      genera in the above sense
  
  \subsection{} \label{ss:genus1} For a natural number $n$, an \emph{ordered system} of \emph{rank} $n$ in $V$ is by definition, an $n$-tuple $\bb=(\b_{1},\ldots, \b_{n})\in V^{n}$. The {\it Gram genus} of  $\bb=(\b_1,\dots,\b_n)$  is defined to be  the Gram matrix
 $C'(\bb):=(c_{i,j})_{i,j=1,\ldots, n}\in \real^{n\times n}$ where $c_{i,j}:=\mpair{\b_{i},\b_{j}}$. 
This defines a $\Sym_{n}$-equivariant  map $C'\colon V^{n}\to \CG^{(n)}$.

 Note that $\s:=C'(\bb)$ is a  positive semidefinite matrix, and is  positive 
 definite (equivalently, it is invertible)  if and only if  $[\bb]$ is linearly independent. More generally, the 
 space $\set{a=(a_{1},\ldots, a_{n})^{t}\in
 \real^{n}\mid \sum_{i=1}^{n}a_{i}\b_{i}=0}$ of linear 
 relations on $\bb$ identifies with the radical of the quadratic
  form $a\mapsto a^{t} C'(\bb)a\colon \real^{n}\to \real$   with matrix $C'(\bb)$ with respect to the standard basis of $\real^{n}$.  It readily  follows that the non-empty fibres of the map $C'$ are the orbits  of the orthogonal group 
  $\mathrm{O}_{V}:=\mathrm{O}(V,\mpair{-,-})$ acting diagonally on $V^{n}$
  (and the set of matrices over which the fibres are non-empty is  precisely the set of  positive   semidefinite, symmetric
matrices in $\real^{n\times n}$). Also,  for $1\leq i<j\leq n$, one has $\b_{i}=\b_{j}$  if and only if the $i$-th and $j$-th  columns (or rows) of $C'(\bb)$ are equal.
In particular,  the columns of $\s$ are pairwise distinct if and only if  $\b_{1},\ldots, \b_{n}$ are pairwise distinct.  In that case,   $G_{\sigma}$ is isomorphic to the group of isometries of the metric space $[\bb]$ (with metric induced by the norm from $\mpair{-,-}$).

\subsection{} Suppose that $\Phi$ is crystallographic.  If $\bb\in \Phi^{n}$ is an 
ordered system of rank $n$ consisting  of roots in $\Phi$, its \emph{Cartan genus} is 
the matrix  of inner products 
$C''(\bb):=(\mpair{\ck \b_i, \b_{j}})_{i,j=1}^{n}$. This  defines 
a $\Sym_{n}$-equivariant  map 
$C''\colon \Phi^{n}\to \CG^{(n)}$.
Again, one has $\b_{i}=\b_{j}$ if and only if the $i$-th and $j$-th columns of $C''(\bb)$ are equal.
Note that $\bb$ is an ordered simple system if and only if  its Cartan genus $C''(\bb)$ is a  Cartan matrix.
 Similarly, $\bb$ is an ordered np subset of rank $n$, 
if and only if $C''(\bb)$ is a generalised Cartan matrix  
(which  then necessarily  has only finite and affine  components). 
  Clearly,  the Cartan genus $C''(\bb)$ of $\bb\in \Phi^{n}$ is completely determined by the Gram genus $C'(\bb)$.
\begin{rem}
If $\bb$ is an ordered simple system in $\Phi$, the automorphism group (see Definition \ref{def:genus}) of
$C''(\bb)$ is known as the group of {\it diagram automorphisms} of  $\Phi$.
\end{rem}

\subsection{} \label{ss:genus2} In \ref{ss:genus2}--\ref{rem:genus}, fix  a subgroup $W'$ of $\mathrm{O}_{V}$ and a $W'$-stable subset $\Psi$ of $V$. The main situation of  interest is that  in which $(W',\Psi)=(W,\Phi)$, but other cases such as when $(W',\Psi)$ is   $(W,V)$ or $(\mathrm{O}_{V},V)$ are also  of  interest. 
  Fix a natural number $n$. Let 
\begin{equation}
\binom{\Psi}{n}:=
  \mset{\Gamma\subseteq \Psi\mid \vert \Gamma\vert =n}\subseteq \CP(\Psi)
\end{equation} 
be the ``configuration space'' of $n$ distinct 
  unordered points of $\Psi$. This has a natural $W'$ 
  action given  by $(w,\Gamma)\mapsto w(\Gamma):=
  \mset{w\gamma\mid \gamma\in \Gamma}$ for
   $w\in W'$ and $\Gamma\in \binom{\Psi}{n}$. 
   
Our main interest will be the 
  study of  $W'$-orbits on $\binom{\Psi}{n}$. With this in mind, define the configuration space 
\begin{equation*}
\Psi^{(n)}:=\mset{\bb=(\b_{1},\ldots, \b_{n})\in \Psi^{n}
\mid \b_{i}\neq \b_{j}\text{ \rm  if $i\neq j$}}
\end{equation*} 
of $n$ ordered distinct points of $\Psi$. 
Then $\Psi^{(n)}$ admits the  diagonal $W'$-action and a commuting  $\Sym_{n}$-action by place permutations, hence  a natural $W'\times \Sym_{n}$-action. 
Moreover, there is  a natural $W'$-equivariant 
surjection $\hat \pi:\bb\mapsto [\bb]\colon \Psi^{(n)}\to  
\binom{\Psi}{n}$. The 
fibres of $\hat \pi$ are 
the $\Sym_{n}$-orbits on $\Psi^{(n)}$, and $\Sym_{n}$ acts simply transitively on each fibre. 
As in \eqref{eq:prodorb}, we may canonically identify the set of $W'$-orbits on $\binom{\Psi}{n}$ as 
\begin{equation}\label{Wconfig}
{\textstyle\binom{\Psi}{n}}/W'\cong
(\Psi^{(n)}/\Sym_{n})/W'\cong  \Psi^{(n)}/(W'\times \Sym_{n})\cong  (\Psi^{(n)}/W')/\Sym_{n}.
\end{equation}

\subsection{} There is a natural $\Sym_{n}$-equivariant map $\bb\mapsto C'(\bb)\colon \Psi^{n}\to \CG^{(n)}$ which assigns to $\bb\in \Psi^{n}$ its Gram genus. 
As a variant in the case $(W',\Psi)=(W,\Phi)$, one can
 consider instead the map $\bb\mapsto C''(\bb)\colon \Psi^{n}\to \CG^{(n)}$
  which assigns to  $\bb$ its Cartan genus; $C''$ is also $\Sym_n$-equivariant.  Let $C\colon \Psi^{n}\to \CG^{(n)}$
be one of  the  maps $C'$ or $C''$. Let $\CG^{(n)}_{0}$ 
  be the subset of $\CG^{(n)}$ consisting of matrices 
  with pairwise distinct columns. 
 As remarked already, for $\bb=(\b_{1},\ldots, \b_{n})\in \Psi^{n}$, one has 
$\bb\in \Psi^{(n)}$ (i.e. $\b_{i}\neq \b_{j}$ for $i\neq j$) if and only if 
one has $C(\bb)\in \CG^{(n)}_{0}$.

 For a genus $\s$ of rank $n$ and type 
$\tau:=I(\s)$, we write $\CS(\s):=C^{-1}(\s)$ for the 
fibre of $C$ over $\s$
and \begin{equation*}
\CT(\tau):=\mset{\bb\in \Psi^{(n)}\mid C(\bb)\in \tau}=\bigcup_{\rho\in \Sym_{n}}\rho(\CS(\s))
\end{equation*} for 
the union of the fibres of $C$ over  all genera of the same 
type as $\s$.  Let \begin{equation*}
\CU(\tau):=\mset{[\bb]\mid \bb\in 
\CT(\tau)}=\mset{[\bb]\mid \bb\in 
\CS(\s)}\subseteq \CP(\Psi).
\end{equation*}
One has $\CU(\tau)\subseteq  \binom{\Psi}{n}$ if and only if $\s\in \CG^{(n)}_{0}$ (or equivalently, $\tau
\subseteq \CG^{(n)}_{0}$).  If it is  necessary to indicate dependence on $C$, we write $\CS_{C}(\tau)$ etc. 

For example, if $(W',\Psi,C)=(W,\Phi,C)$ and  $\tau$ is a type of Cartan matrices,
 then
$\CT(\tau)$ (resp., $\CU(\tau)$) is the set of all  ordered 
(resp., unordered) simple subsystems of $\Phi$ of that type. Similarly if $\tau$ is a type of generalised Cartan matrices, then $\CT(\tau)$ (resp., $\CU(\tau)$) is the set of ordered (resp., unordered) np subsets of that type (this set being empty unless all components of $\tau$ are of finite or affine type). 

In general, each  set $\CS(\s)$, $\CT(\tau)$ and $\CU(\tau)$ is $W'$-stable.  In particular,  the classification of $W'$-orbits on $\binom{\Psi}{n}$ is   reduced to the classification of $W'$-orbits on $\CU(\tau)$ for each type of  genus $\tau\subseteq \CG^{(n)}_{0}$ (with $\CU(\tau)\neq \eset$).  We describe an approach to this classification  along lines 
similar to \eqref{Wconfig},
  by study of $\CS(\s)$ and $\CT(\s)$

\subsection{} \label{ss:genus3} Let $\s\in \CG^{(n)}$ and
$\tau:= I(\s)$.
One has a commutative diagram\begin{equation*}
\xymatrix{{\CS(\s)}\ar@{^(->}[r]\ar[d]_{\pi_{\s}}&{\CT(\tau)}\ar[d]^{\pi}\\
{\CU(\tau)}\ar@{=}[r]&{\CU(\tau)}
}
\end{equation*}
 of $W'$-sets  
in which the top horizontal map is an inclusion, and  $\pi$ and $\pi_{\s}$ are the indicated restrictions of the map $\hat\pi\colon \bb\mapsto [\bb]$.   Since $C$ is 
$\Sym_{n}$-equivariant and 
$\tau=I(\s)$ is the $\Sym_{n}$-orbit of $\s$, $\Sym_{n}$ acts naturally on $\CT(\tau)$, commuting with its $W'$-action. In this way, $\CT(\tau)$ acquires a natural  structure of $W'\times \Sym_{n}$-set. By restriction, $\CS(\s)$ has a natural structure of  $W'\times G_{\sigma}$-set.
This $W'\times G_{\s}$-set $\CS(\s)$ depends only on the type of $\s$ up to natural identifications. More precisely, 
let $\s'=(a_{i,j}')_{i,j=1,\ldots, n}$ be another genus of  type $\tau$, say $\s'=\rho\s$  where $\rho\in \Sym_{n}$. Then $G_{\s'}=\rho G_{\s}\rho^{-1}$ and $\CS(\s')=\rho \CS(\s)$; in fact, the map
$ p\colon  \CS(\s)\to \CS(\s')$ given by $\bb\mapsto \rho \bb$ is a  $W'$-equivariant bijection and satisfies  $p(\nu\bb)=(\rho\nu\rho^{-1})p(\bb)$ for $\bb\in \CS(\bb)$ and $\nu\in G_{\s}$.

\subsection{} Assume henceforth that $\s\in \CG^{(n)}_{0}$, so
 $\tau\seq \CG^{(n)}_{0}$.
Then the 
$\Sym_{n}$-orbits on  $\CT(\tau)$ (resp, $G_{\s}$-orbits on $\CS(\s)$) are precisely the fibres of $\pi$ (resp., $\pi_{\s}$) and  $\Sym_{n}$ (resp., $G_{\s}$)
acts simply transitively on each fibre of $\pi$ (resp., $\pi_{\s})$. 
(There is  even  a natural isomorphism of $\mathrm{W'}\times \Sym_{n}$-sets $\Sym_{n}\times^{G_{\s}}\CS(\s)\cong \CT(\t)$.)  
Hence we may  naturally identify 
$\CU(\tau)\cong \CT(\tau)/\Sym_{n}\cong \CS(\s)/G_{\s}$ as $\mathrm{W'}$-sets. From
 \eqref{eq:prodorb}, one gets canonical identifications\begin{equation}\label{eq:orbits}\CU(\tau)/W'\cong (\CT(\tau)/W')/
\Sym_{n}\cong (\CS(\s)/W')/
G_{\s}.
 \end{equation}

 In the case $(W',\Psi,C)=(\mathrm{O}_{V},V,C')$, each non-empty set $\CS(\s)$ for $\s\in \CG^{(n)}_{0}$ is a single $W'$-orbit on $\Psi^{(n)}$ (see \ref{ss:genus1}) and thus $\CU(\tau)$ is a $W'$-orbit on $\binom{\Psi}{n}$. The above 
 equations and the discussion of \ref{ss:genus1} therefore  give rise to the following closely related parameterisations of the set of 
 $\mathrm{O}_{V}$-orbits of unordered sets of $n$ (distinct) points in $V$: the set of such orbits  corresponds bijectively to the set of   symmetric, 
 positive semidefinite real $n\times n$ matrices with distinct columns, 
 modulo simultaneous row and column permutation
 (resp., to a set of representatives of such matrices under such permutations).

\begin{remark}\label{rem:genus} For any groups $G$, $H$ there is a natural analogue in the category of $G$-sets of the standard notion of principal $H$-bundle.  Transferring standard terminology from the  theory of principal bundles to this setting,    the above shows that (assuming  $\s\in \CG^{(n)}_{0}$), $\wh \pi$ and  $\pi$ 
     are     principal $\Sym_{n}$-bundles and $\pi_{\s}$ is a 
     principal $G_{\s}$-bundle which affords a  reduction of the structure group of the  bundle $\pi$ from $\Sym_{n}$ to $G_\s $.
     Moreover,  the bundle $\pi_{\s}$ depends 
(up to natural identifications of the various possible structure groups $G_{\s}$ induced by inner automorphisms of $\Sym_{n}$) only on the type of $\s$ and not on $\s$ itself. 
    Simple examples show that in general, these bundles are not trivial bundles.   \end{remark}

\subsection{}\label{ss:genus4} From now on,  we restrict to the case $W'=W$, which is of primary interest here. 
Recall the definition of the  dot action of $\Sym_{n}$ on the fundamental region $\CC_{W}^{(n)}$: for  $\rho\in \Sym_{n}$ and $\bb\in \CC^{(n)}_{W}$, one has $\set{\rho \cdot \bb}= W\rho \bb\cap\CC_{W}^{(n)}$.  Recall also 
the maps $C',C''$ which associate to an $n$-tuple of roots its associated Gram genus and Cartan genus respectively.

  \begin{prop}\label{prop:conjreps} Let $\s\in \CG^{(n)}_{0}$ and $\t:=I(\s)$.
 The  set  
$\CU(\tau)/W$ of $W$-orbits on $\CU(\tau)$ may be
canonically identified with $\CS(\s)/G_{\s}$ where $\CR_C(\s)=\CR(\s):=\CS(\s)\cap \CC^{(n)}_{W}$  (resp.,  with
$ (\CT(\tau)\cap \CC^{(n)}_{W})/
\Sym_{n}$)
 and the  $G_{\s}$ (resp., $\Sym_{n}$-action) used to form the coset space is a restriction of  the dot action.
  \end{prop}
 \begin{proof}
 This follows immediately from 
 \eqref{eq:orbits} (with $W'=W$) on identifying $\CT(\tau)/W$ with $\CC_{W}^{(n)}\cap \CT(\tau)$ and 
$\CS(\s)/W$ with $\CC_{W}^{(n)}\cap\CS(\s)$ (the corresponding actions by $\Sym_{n}$ and $G_{\s}$ identify with restrictions of the dot action). 
 \end{proof}

     \subsection{}  We now    record some consequences  of the  bijection $\CU(\tau)/W\cong \CR(\s)/G_{\s}$ in the case $(W',\Psi)=(W,\Phi)$, for the classification of $W$-orbits of simple subsystems of  $\Phi$; similar results hold for $W$-orbits of  arbitrary subsets of roots (or indeed, of   vectors in $V$).   For any $\bb\in \Phi^{(n)}$, we call $I(C'(\bb))$
     (resp., $I(C''(\bb))$)  the \emph{Gram type} (resp., \emph{Cartan type})  of $\bb$ and of $[\bb]\in \binom{\Phi}{n}$.   Note that specifying the Gram type of a
     (unordered or ordered) simple system  amounts to   specifying its Cartan type (i.e. its Cartan matrix up to reindexing) together with the function taking each irreducible component to the maximal root length in that component. 
     
     By a Cartan (resp., Gram) genus of simple systems we mean the Cartan (resp.,  Gram) genus of some ordered simple system of some crystallographic (resp., arbitrary) finite root system.  Thus, a Cartan genus of ordered simple systems  is just a Cartan matrix $(a_{i,j})_{i,j=1}^{n}$ (with a specified indexing by $1,\ldots, n$).

%      As observed above,  the automorphism group $G_\s $ acts 
% naturally on the set  $\CS(\s)$ of all simple systems of genus $\s$. 
% There is then an induced dot $G_\s $-action on $\CR(\s):=\CS(\s)
% \cap \CC^{(n)}_{W}$ as in Proposition \ref{prop:genclass}. 
% The following result is the key to the use of the  results contained in Section \ref{s1} in studying 
%$W$-classes of subsystems.

 \begin{prop}\label{prop:genconj} Let $\s$ be a Cartan (resp., Gram) genus of  ordered simple systems
 and $C:=C''$ (resp., $C:=C'$). Let   $\tau:=I(\s)$.  \begin{num}
 \item  There is a natural bijection between $W$-orbits of simple systems of Cartan (resp., Gram) type $\tau$ in $\Phi$ and $G_\s $-orbits for the dot action on $\CR_{C}(\s)$.
  \item The conjugacy classes of simple subsystems of Cartan (resp., Gram) type $\tau$  are exactly the  conjugacy classes of simple systems $[\bb]$ for $\bb\in \CR_{C}(\s)$. 
 \item The number of $W$-orbits of simple systems of Cartan (resp., Gram) type $\tau$ in $\Phi$ is at most the number  $\vert \mset{[\bb]\mid \bb\in \CR_{C}(\s)}\vert$
 of simple systems $[\bb]$  which  underlie some  ordered simple system $\bb$ in $\CR_{C}(\s)$.
 \item If $\vert \CR_{C}(\s)\vert=1$, or, more generally, $\CR_{C}(\s)\neq \eset$ and all ordered simple systems $\bb$ in $\CR_{C}(\s)$ have the same underlying simple system $[\bb]$, then there is a single $W$-conjugacy class of simple systems in $\Phi$ of Cartan (resp., Gram) type $\tau$.
    \end{num}
  \end{prop}
  \begin{proof}  Since  $\CS_{C}(\s)$ is the set of ordered  simple systems of $\Phi$ with Cartan (resp.,
   Gram) matrix $\s$, part (a) follows from Proposition 
   \ref{prop:conjreps}.  We prove (b).  Note that if $\D$ is a simple subsystem of Cartan (resp.  Gram) type 
      $\tau$, then $\Delta=[\bb]$ for some $\bb\in \CS(\s)$. We have  $w\bb\in \CR(\s)$ for some $w\in W$. 
      Thus, $w\Delta=[w\bb]$. So the $W$-orbit of any 
      simple subsystem $\Delta$ of Cartan (resp., Gram) type $\tau$ has a 
      representative $w\Delta$ in $\CR_{C}(\s)$. On the other hand,  for any $\bb\in \CR(\s)$,  one has $[\bb]\in \CT(\tau)$. Part (b) follows. 
Part (c) follows directly from (b), and part (d) follows from (c).
     \end{proof}
     \begin{remark} For fixed $\Phi$ and any Cartan genus $\s$ of ordered simple systems of type $\tau$, 
there are finitely many Gram genera $\s_{1},\ldots, \s_{n}$ of types $\tau_{1},\ldots, \tau_{n}$ respectively, such that
     $\CS_{C''}(\s)=\disjun_{i=1}^{n}\CS_{C'}(\s_{i})$,
     $\CR_{C''}(\s)=\disjun_{i=1}^{n}\CR_{C'}(\s_{i})$
     $\CT_{C''}(\t)=\cup_{i=1}^{n}\CT_{C'}(\t_{i})$ and
     $\CU_{C''}(\t)=\cup_{i=1}^{n}\CU_{C'}(\t_{i})$.
     Gram genera are convenient in that there are typically fewer simple systems of a specified Gram genus than a corresponding Cartan genus; 
on the other hand, it may be inconvenient  to specify the several Gram genera necessary to describe all (ordered) simple systems of a specified 
Cartan genus (e.g.  an isomorphism type of simple subsystems).      
     \end{remark}

% 
%\begin{remark} Suppose that $\sigma$ is a genus  which is of the 
%  same type as a simple system  $\Delta$. An assignment of $\Delta$ to 
%  $\sigma$ is defined to be a bijection $f\colon \set{1, \ldots,\vert \Delta\vert}  \to \Delta$
%  such that $(f(1),\ldots, f(\vert \Delta\vert)\in \CS(\s)$.
% Define a  marking of $\Delta$ to be an assignment to $\Delta$   of   additional data of the following types (1)--(2): (1)  for each  isomorphism class 
%$X$  of finite crystallographic root 
%system,  a total ordering of the set of components of $\Delta$ 
%of type $X$ (2) for each component $\Delta'$ of $\Delta$ with non-trivial automorphism group $G_{\Delta'}$ (of order $(k+1)!$, say)
%an injective function $\set{1,\ldots, k}\to O$  where $O\subseteq \D'$ is the (unique) $G_{\Delta'}$ orbit of cardinality $k+1$ on the  set of terminal vertices of the diagram of $\Delta'$.   One may establish for fixed $\Delta$  a (non-canonical) bijection between the set  of assignments of $\Delta$ to the genus $\s$ and the sets of markings of $\Delta$.\end{remark}   

As an application, we shall
show that if $L$ is  the length of some root of $\Phi$,  there is a unique conjugacy class of simple subsystems of Gram type $I(L\cdot \Id_{n})$ 
with $n$ maximal. 

\begin{prop}\label{prop:simply}
Let  $L\in \real_{>0}$. Consider simple subsystems $\Delta$ of $\Phi$ satisfying the three conditions below: 
\begin{conds}
\item $\Delta$ is of type $A_{1}^{n}$ for some $n$. 
\item All roots in $\Delta$ have the same length $L$.
\item $\Delta$ is inclusion maximal subject to 
{ \rm (i)} and { \rm (ii)}.
\end{conds} Then any two  subsystems  of $\Phi$ 
satisfying { \rm (i)--(iii)}  are conjugate under $W$.
\end{prop}
\begin{proof} Following Proposition
\ref{prop:genclass}, we begin by describing
all the ordered simple subsystems 
$\bb=\mset{\b_1,\dots,\b_n}$  in 
$\CC^{(n)}_W$ satisfying (i)--(ii).  
First, $\b_1$ is a dominant root of 
length $L$ of some component of $\Phi$.
The possibilities for $\b_2$ are identified by deleting 
from the affine Dynkin diagram
of $\Pi\disjun\{-\b_1\}$ the root $-\b_1$ and all roots 
attached to $-\b_1$. One then 
obtains a simple system for a  new root (sub)system 
and repeats the above process. At each stage, the only
 choice
is in which component (containing roots of length $L$)
 of the 
relevant root system one takes the dominant root of 
length $L$.
The condition (iii) will be satisfied if and only if 
 none of the components at that 
stage contain roots of length $L$. 
By the maximality assumption (iii),    the integer $n$ 
and the set $[\bb]=\mset{\b_1,\dots,\b_n}$ is clearly 
independent of the order in which these operations
 are done. 

This implies that all  ordered simple systems which satisfy (i)--(iii) and which lie  in the fundamental region  are (certain) permutations of each other. 
It follows by Proposition
\ref{prop:genconj}(c),  there is just one conjugacy class of unordered subsystems satisfying (i)--(iii).
\end{proof}

  \subsection{Action of the longest element}\label{aut}   For any simple subsystem $\Delta$ of   
  $\Phi$, we denote by $\omega_{\Delta}$ the longest element of the 
  Coxeter system $(W_{\Delta},S_{\Delta})$
  where $S_{\Delta}=\mset{s_{\alpha}\mid \alpha\in \Delta}$ and  
  $W_{\Delta}:=\mpair{S_{\Delta}}$. It is well known that there is 
  some permutation $\rho_{\Delta}$ of $\Delta$ (which may and often  will be regarded as  a diagram automorphism of the Dynkin diagram of $\Delta$) such that $\omega_{\Delta}(\alpha)=-\rho_{\Delta}(\alpha)$ for all $\a\in \Delta$.
  In case $\Delta=\Pi$, we write $\omega_{\Pi}=\o$ and $\rho_{\Pi}=\rho$. Extending the permutation action on $\Pi$ linearly, we
  sometimes  regard $\rho_{\Pi}$ as an automorphism of $\Phi$ which fixes the set $\Pi$. If $\bb$ is an ordered simple subsystem of $\Phi$, then $[\bb]$ 
  is a simple subsystem of $\Phi$ and so $\o_{\bb}:=\o_{[\bb]}$ and $\rho_{\bb}:=\rho_{[\bb]}$ are defined as above.

Recall from Proposition \ref{prop:conjreps} that the set $\CR_{C}(\s)$ of ordered np sets of genus $\s$ in the fundamental region 
 has a $\Sym_n$ dot action. In general, this action may not be easy to describe explicitly.
However,  the following useful special result gives a
simple description of the effect of applying  diagram
automorphisms induced by longest elements, to 
simple subsystems  in the fundamental region for $W$.

  \begin{prop}\label{prop:diagaut} Let $\bb:=(\b_{1},\ldots, \b_{n})\in \Phi^{n}\cap \CC^{(n)}_{W}$ be an ordered simple system (for some root subsystem $\Psi$ of $\Phi$) which lies  in the fundamental region for $W$ on $V^{n}$. Then $W(\rho_{\bb}(\b_{1}),\ldots, \rho_{\bb}(\b_{n}))\cap \CC_{W}^{(n)}=\set{(\rho(\b_{1}),\ldots, \rho(\b_{n}))}$ i.e. 
  $(\rho(\b_{1}),\ldots, \rho(\b_{n}))$ is the representative, in the fundamental region for $W$ on $V^{n}$, of the ordered simple subsystem 
  $(\rho_{\bb}(\b_{1}),\ldots, \rho_{\bb}(\b_{n}))$ of $\Psi$.
    \end{prop}
    \begin{proof}
    One has $\o_{\bb}(\b_{1},\ldots, \b_{n})=-(\rho_{\bb}(\b_{1}),\ldots, \rho_{\bb}(\b_{n}))$ and therefore  
    $(\b_{1},\ldots, \b_{n})=-\o_{\bb}(\rho_{\bb}(\b_{1}),\ldots, \rho_{\bb}(\b_{n}))\in \CC_{W}^{(n)}$ since $\o_{\bb}$ is either an involution or the identity. 
    By \ref{ss:diagaut},  $\CC_{W}^{(n)}\cap \Phi^{n}$ is $-\o$-invariant. Hence  \begin{equation*}
    -\o(\b_{1},\ldots, \b_{n})=
    (\o\o_{\bb})(\rho_{\bb}(\b_{1}),\ldots, \rho_{\bb}(\b_{n}))\in \CC_{W}^{(n)}\cap \Phi^{n}.
    \end{equation*} Since $\o\o_{\bb}\in W$ and $-\o(\b_i)=\rho(\b_i)$, the proposition follows from Theorem \ref{thm:fundreg}.
        \end{proof}
  \section{Simple subsystems of type $A$}
  \label{sec:typeA}
  In this section, $\Phi$ denotes a crystallographic root system. We shall define a standard genus $\s_{n}$ of type $A_{n}$ and determine all ordered simple subsystems of $\Phi$ of Cartan genus $\s_{n}$ in the fundamental region for $W$. 
 The following result of Oshima plays an important role in the proof of  this  and related results in arbitrary type.

\begin{lemma} \label{lem:oshorig} \cite{Osh} Suppose that  $W$ is an irreducible finite Weyl group with 
crystallographic root system $\Phi$.  Fix $\Delta\seq \Pi$, and scalars $c_{\b}\in \real $ 
for $\b\in \Delta$, not all zero, and $l\in \real$.
 Let $\g[\a]$ denote the coefficient of $\a\in\Pi$ in $\g$, expressed as a linear combination of $\Pi$,
and write
\begin{equation*}
X:=\mset{\g\in \Phi\mid \mpair{\g,\g}^{1/2}=l \text{ \rm and } \g[\b]=c_{\b} \text{ \rm for all $\b\in \D$}}.
\end{equation*}  
If $X$ is non-empty, then it is a single   $W_{\Pi\sm \D}$-orbit of roots.  
Equivalently, $\vert X\cap \CC_{\Pi\sm \D}\vert \leq 1$.
\end{lemma}
This is proved in \cite[Lemma 4.3]{Osh} using facts in \cite{OdaOsh} from  the 
representation theory of  semisimple complex Lie algebras. 
We shall give a direct proof of a more general version of  Lemma \ref{lem:oshorig} in a future work \cite{DyLOsh}, but in this paper 
make use of it only as stated.

 \begin{definition}\label{def:genusa}
 The  standard genus  
 $\s_{n}$ of Cartan  type $A_{n}$ is the Cartan matrix of type $A_{n}$ in its standard ordering as in \cite{Bour} i.e.
 $\s_{n}=(a_{i,j})_{i,j=1}^{n}$ where for $1\leq i\leq j\leq n$, one has  $a_{ij}=2$ if $i=j$, $a_{i,j}=a_{j,i}=-1$ if $j-i=1$ and $a_{i,j}=a_{ji}=0$ if $j-i>1$.  
\end{definition}
 
 Thus, for $\bb:=(\b_{1},\ldots, \b_{n})\in \Phi^{n}$, 
 one has $\bb\in \CS_{C''}(\s_{n})$ if and only if  one has $(\mpair{\ck\b_{i},\b_{j}})_{i,j=1,\ldots,n}=\s_{n}$.  
 Clearly, the automorphism group $G_{\s_{n}}$ of $\s_{n}$ is trivial if $n=1$ and is cyclic of 
 order $2$ if $n>1$, with non-trivial element the longest 
 element of $\Sym_{n}$ in the latter case.
 Recall that $\CR_{C''}(\s_{n}):=\CS_{C''}(\s_{n})\cap\CC_{W}^{(n)}$. \begin{prop}\label{prop:Agenus} 
Let $\bb=(\b_{1},\ldots, \b_{n})\in \CR_{C''}(\s_{n})$ where $n\geq 1$. 
Then 
  $\b_{1}$ is  a dominant root of $\Phi$ and $\b_{2},\ldots, \b_{n}\in -\Pi$.  \end{prop}
\begin{proof}
The proof is by induction on $n$. Since $\mpair{\ck\b_{i},\b_{i+1}}=-1$ for $i=1,\ldots, n-1$, $\b_{2},\ldots, \b_{n}$ lie in the same component of $\Phi$ as $\b_{1}$. 
Without loss of generality, we may replace $\Phi$ by the component containing $\b_{1}$ and assume that $\Phi$ is irreducible. For $n=1$, one has 
$\b_{1}\in \Phi\cap\CC_{W}$ and the result is clear. Now suppose that 
$n>1$. Since 
$\bb':=\tau_{n-1}(\bb)=(\b_{1},\ldots, \b_{n-1})\in \CR_{C''}(\s_{n-1})$, induction gives $\b_{2},\ldots, \b_{n-1}\in -\Pi$.
Let  $\G:=\Pi\cap \set{\b_{1},\ldots, \b_{n-1}}^{\perp}$ and 
$\D:=\Pi\cap \set{\b_{1},\ldots, \b_{n-2}}^{\perp}\sreq \G$. Then
$W_{\G}=W_{\bb,n-1}$ and $W_{\D}=W_{\bb,n-2}$ by Theorem \ref{thm:fundreg}.
One has $s_{\b_{n}}\in W_{\bb,n-2}=W_{\D}$ since $\mpair{\b_{n},\b_{i}}=0$ for $i<n-1$, so  $\Sigma:=\supp(\b_{n})\seq \D$.
Define \begin{equation*}
\L:=\mset{\g\in \Sigma\mid \mpair{\g,\b_{n-1}}\neq 0}.
\end{equation*}
Since $0\neq \mpair{\b_{n},\b_{n-1}}=\sum_{\g\in \Sigma}\b_{n}[\g]\mpair{\g,\b_{n-1}}$, it is clear that $\vert \L\vert\geq 1$. 

We claim that $\vert \L\vert=1$.  Suppose to the contrary that 
$\vert \L\vert\geq 2$.   We have $\b_{n-1}\not \in \D$. 
If $n\geq 3$, then the full subgraph of the Dynkin diagram on vertex set
 $ \Sigma  \cup\set{-\b_{n-1}}\seq \Pi$ contains a cycle, contrary to 
 finiteness of $W$. Hence $n=2$, $\b_{n-1}\in \Phi\cap \CC_{W}$, 
  $\D=\Pi$ and 
$\vert \set{\g\in \Pi\mid \mpair{\b_{n-1},\g}>0}\vert \geq \vert \L\vert \geq 
2$. The classification implies that $\Phi$ is of type $A_{m}$ with $m>1$. Regard
$\Pi\cup\set{-\b_{1}}$ as the vertex set of the affine Dynkin diagram 
corresponding to $\Phi$.
 Then the full subgraph on $\set{-\b_{1}}\cup  \Sigma  $ contains a cycle,
so its vertex set  is equal to  $\Pi\cup\set{-\b_{1}}$. This implies that 
$ \Sigma  =\Pi$ and  $\b_{2}=\b_{n}=\pm \b_{1}$ contrary to 
$\mpair{\b_{n-1},\ck\b_{n}}=-1$. This proves the claim.

Write $\L=\set{\a}$ where $\a\in \Pi$ and abbreviate $c:=\b_{n}[\a]\neq 0$. We have
\begin{equation*}
\mpair{\b_{n},\b_{n-1}}=\sum_{\g\in \Sigma}\b_{n}[\g]\mpair{\g,\b_{n-1}}=
\b_{n}[\a]\mpair{\a,\b_{n-1}}=c\mpair{\a,\b_{n-1}}. 
\end{equation*} 
Writing $\mpair{\b_{n},\b_{n}}=\mpair{\b_{n-1},\b_{n-1}}=d\mpair{\a,\a}$,
it follows that
$-1=\mpair{\b_{n},\ck\b_{n-1}}=c\mpair{\a,\ck\b_{n-1}}$. Since $c,\mpair{\a,\ck\b_{n-1}}\in \Int$, this implies $c=-\mpair{\a,\ck\b_{n-1}}\in\set{\pm 1}$.
We have $ \mpair{\a,\ck\b_{n-1}}\geq 0$ since $\a\in \Pi$ and  $\b_{n-1}$ is a dominant root if $n=2$, and $-\b_{n-1}\neq \a$  are both  in $\Pi$ if $n>2$.  
Hence  $c=\mpair{\a,-\ck \b_{n-1}}=-1 $. Also,  
 $-1=\mpair{\b_{n-1},\ck \b_{n}}=\frac{c}{d}\mpair{\b_{n-1},\ck\a}$
 so $\mpair{\b_{n-1},\ck\a}=-dc^{-1}=d\in \Int_{\geq 1}$.
 From 
 \begin{equation*}
 \ck\b_{n}=\frac{2\b_{n}}{\mpair{\b_{n},\b_{n}}}=\frac{2}{\mpair{\b_{n},\b_{n}}}\sum_{\g\in \Pi}\b_{n}[\g]\g=\sum_{\g\in \Pi}\frac{\mpair{\g,\g}}{\mpair{\b_{n},\b_{n}}}\b_{n}[\g]\ck \g
 \end{equation*}
 one sees that $\Int\ni\ck\b_{n}[\ck \a]=\frac{\mpair{\a,\a}}{\mpair{\b_{n},\b_{n}}}\b_{n}[\a]=\frac{c}{d}=-d^{-1}$ so $d=1$.
 
 We have now established that $\mpair{\a,\a}=\mpair{\b_{n-1},\b_{n-1}}=\mpair{\b_{n},\b_{n}}$. This implies $\mpair{\b_{n-1},\ck \a}=\mpair{\a,\ck
  \b_{n-1}}=1$. 
  To complete the proof, it will suffice to  show that
   $\b_{n}=-\a$.  Taking $\Sigma=\supp(\b_{n})$ as above, let $\Psi:=W_{\Sigma}\Sigma$; this is an irreducible crystallographic root system.  
From above, one has $\b_{n},-\a\in \Psi$, $\b_{n}[\a]=-1=(-\a)[\a]$ and
   $\mpair{\b_{n},\b_{n}}=\mpair{\a,\a}$.  By Lemma \ref{lem:oshorig}, it follows that
   $\b_{n}$ and $-\a$ are in the same $W_{\Sigma'}$-orbit on $\Phi_\Sigma$ where 
   $\Sigma':=\Sigma\sm\set{\a}$. 
    But  $\Sigma'=\Sigma\sm \L=
   \Sigma\cap \b_{n-1}^{\perp}\seq   \D\cap \b_{n-1}^{\perp}=\G$. 
   Hence $\b_{n}\in \CC_{W_{\G}}\seq \CC_{W_{\Sigma'}}$.
   We also have $-\a\in \CC_{W_{\Sigma'}}$. From Lemma 
   \ref{lem:fundcham},    it follows that $\b_{n}=-\a$ as required.
\end{proof}

The following Corollary \ref{prop:ABij} and Theorem \ref{thm:Aconj} are closely related to the type $A_{n}$ case of the result \cite[Theorem 3.5(i)]{Osh} of  Oshima, while  Remark \ref{rem:graphical} is related to that result in some other classical types.

\begin{corollary}\label{prop:ABij} Let $P_{n}$ denote the set of all pairs $(\beta, \Gamma)$ such that $\beta$ is a dominant root of $\Phi$, $\Gamma$ is the vertex set of a  type $A_{n}$ subdiagram  of the Dynkin diagram of $\set{\beta}\cup-\Pi$ and $\beta$ is a terminal vertex of $\G$. Then there is  a bijection $\CR_{C''}(\s_{n})\to P_{n}$ defined by  $\bb=(\b_{1},\ldots, \b_{n})\mapsto (\b_{1}, [\bb])$ for all $\bb\in
\CR_{C''}(\s_{n})$.\end{corollary}
\begin{proof} Let $\bb\in \CR_{C''}(\s_{n})$. Then Proposition \ref{prop:Agenus} implies that $[\bb]$ is the vertex set of a subdiagram of the Dynkin diagram
of $\set{\b_{1}}\cup -\Pi$.  By definition of $\s_{n}$, this diagram is of type $A_{n}$ with $\beta_{1}$ terminal and $\beta_{i}$ joined to $\b_{i+1}$ for $i=1,\ldots, n-1$. Conversely, suppose $(\beta,\Gamma)\in P_{n}$.
Since $\Gamma$ is of type $A_{n}$ with $\b$ terminal, we may uniquely write $\G=[\bb]$ for some $\bb=(\b_{1},\ldots, \b_{n})\in \CS(\s_{n})$ with  $\b_{1}=\b$. 
We have to show that $\bb\in \CC_{W}^{(n)}$
i.e. for $i=1,\ldots, n$, $\b_{i}\in\CC_{W_{\Pi_{i}}}$.
where  $\Pi_{i}:=\Pi\cap \set{\b_{1},\ldots, \b_{i-1}}^{\perp}$. One has $\b_{1}\in  \CC$ by the assumption $\b_{1}$ is dominant.  To show  
$\b_{i}\in W_{\Pi_{i}} $ for $2\leq i\leq n$, we have  to show that $\mpair{\b_{i},\g}\geq 0$ for $\g\in \Pi_{i}$. 
But $\g\neq \b_{i}$ since $\mpair{\b_{i},\b_{i-1}}< 0$. Hence  $\mpair{\g,\b_{i}}\geq 0$, since $\g\in \Pi$ and $\b_{i}\in -\Pi$. 
\end{proof}
\begin{rem}\label{rem:graphical} An argument similar to that in the last 
part of the proof shows the following:   Let $\b_{1}\in \Phi$ be a dominant 
root and $\b_{2},\ldots, \b_{n}\in -\Pi$ be pairwise 
distinct and   such that for each $i=2,\ldots, n$, 
there is some $j<i$ with $\mpair{\b_{i},\b_{j}}\neq 0$.
Then $(\b_{1},\ldots, \b_{n})$ is an ordered np subset of $\Phi$ which lies in $\CC_{W}^{(n)}$.\end{rem}

\subsection{Conjugacy of simple subsystems of type $A$}
Let $\rho$ denote  the automorphism   $\rho=\rho_{\Pi}=-\o_{\Pi}$ of 
$\Phi$ induced by the action of the negative of the longest 
element $\o_{\Pi}$ of $W$ (see \ref{aut}). One has 
$\rho^{2}=\Id_{\Phi}$ and $\rho$ fixes each dominant root. Hence $H:=\mpair{\rho}=\set{1,\rho}$ acts naturally on $P_{n}$,  where $P_n$ is  defined in Corollary \ref{prop:ABij}.  
\begin{theorem}\label{thm:Aconj}
 The set of $W$-conjugacy classes of simple systems of type $A_{n}$ in $\Phi$ is in bijection with $P_{n}/H$.
\end{theorem}
\begin{proof} Let $\s=\s_{n}$ and $G=g_{\s}$. By Proposition \ref{prop:genconj}, 
conjugacy classes of simple systems of type $A_{n}$
 in $\Phi$ correspond bijectively to $G$-orbits in the dot
  action of $G$ on $\CR_{C''}(\s)$.  Now $G=\set{1,\th}\subseteq \Sym_{n}$ where 
$\th(i)=n-i+1$ for $i=1,\ldots, n$. For 
$\bb=(\b_{1},\ldots, \b_{n})\in \CR_{C''}(\s_{n})$, one 
has 
$\th(\b_{1},\ldots, \b_{n})= (\b_{n},\ldots, 
\b_{1})=(\rho_{\bb}(\b_{1}),\ldots, \rho_{\bb}(\b_{n}))$ 
for the diagram automorphism $\rho_{\bb}$ defined as 
in \ref{aut}.   For $\bb\in \CR(\s)$,  
  Proposition \ref{prop:diagaut}
implies  that $\th\cdot\bb=\rho(\bb)$ and  so $G\cdot \bb=H\bb$. 
This shows that   conjugacy classes of simple systems 
of type $A_{n}$ in $\Phi$ correspond bijectively to $H$-orbits  on $\CR_{C''}(\s)$.  Finally, the theorem follows on observing that the bijection $ \CR_{C''}(\s)\xrightarrow{\cong} P_{n}$ of Proposition \ref{prop:ABij} is  $H$-equivariant.
\end{proof}

\section{Simple subsystems of arbitrary  types}
\label{sec:arbtype}
In this section, we  indicate, without proofs, how Propositions \ref{prop:genconj} and \ref{prop:Agenus} provide an approach to studying conjugacy of simple subsystems of arbitrary types in finite crystallographic root systems. We confine the discussion to Cartan genera, though  similar results apply to Gram genera and may be more convenient in practice. Throughout this section, $\Phi$ is a (possibly reducible) crystallographic root system.

\subsection{Standard genera of irreducible type} Choose for each type $X_{n}$ of finite irreducible crystallographic root system a standard genus $\s_{X_{n}}=(a_{i,j})_{i,j=1^{n}}$ as follows.  Let $l$ be the maximal rank of a type $A_{l}$ standard 
 parabolic subsystem of a simple system of type $X_{l}
 $ (one has $l=n-1$ unless $X_{n}=A_{n}$, when
  $l=n$, or $X_{n}=F_{4}$, when $l=n-2=2$). We 
  require first that $\s_{X_{n}}$ be a Cartan matrix of 
  type $X_{n}$, that $(a_{i,j})_{i,j=1^{l}}=\s_{l}$ and  that 
  for each $i=l+1,\ldots, n$, there is some (necessarily unique) $p_{i}<i$ 
  such that $a_{p_{i},i}\neq 0$.  Second, we require that
 $\s_{X_{n}}$  be  chosen so $(p_{l+1},\ldots, p_{n})$ is  lexicographically greatest for all possible Cartan 
 matrices satisfying these requirements (i.e. the index 
 of any branch vertex is taken as large as possible 
 subject to the previous conditions). Finally, if 
 $\s_{X_{n}}$ is not uniquely determined by the 
 previous conditions (that is, if it is $B_{2}=C_{2}$, 
 $F_{4}$ or $G_{2}$), we require that $a_{l,l+1}=1$ (i.e. a root corresponding to $a_{l}$ is longer than that 
 corresponding to $a_{l+1}$). In types $A_{n}$ with 
 $n\geq 1$, $B_{n}$  with $n\geq 2$, $C_{n}$ with 
 $n\geq 3$,   $D_{n}$ with $n\geq 4$ and $F_{4}$, 
 $\s_{X_{n}}$ is the Cartan matrix of type $X_{n}$ 
 indexed exactly as in \cite{Bour}; in other types the 
 indexing differs.
 
 \subsection{} The above conditions are chosen on heuristic grounds 
 to make determination of $\CR_{C''}(\s_{X_{n}})$ as  
 computationally simple as possible.  They assure that
  if  $\bb=(\b_{1},\ldots, \b_{n})\in \CR_{C''}(\s_{X_{n}})$ 
  then at most one of $-\b_{2},\ldots, -\b_{n}$ is not  a simple root (except for $X_{n}=F_{4}$ 
   which requires a separate  check, this follows from  
   Proposition  \ref{prop:Agenus} since $\tau_{l}
   \bb=(\b_{1},\ldots, \b_{l})\in  \CR_{C''}(\s_{A_{l}})$ and
    $l\geq n-1$). 
    
If  such a non-simple root amongst $-\b_{2},\ldots, -\b_{n}$ does not exist, then
 $\bb$ must arise as in Remark \ref{rem:graphical} and the possible such $\bb$ may be determined by inspection of the Dynkin diagram.   It turns out that if there is a non-simple 
 root in $-\b_{2},\ldots, -\b_{n}$, then,  with two types of  exceptions,
it  is   a 
dominant root  for an  irreducible 
standard parabolic subsystem of the component of
 $\Phi$ containing $\b_{1}$ (and thus containing $[\bb]$).  (The exceptions occur when $X_{n}=G_{2}$ and that component is  of type $G_{2}$, or when $X_{n}=E_{6}$ and that component is of type $E_{8}$.)
  In fact,  general results (depending on Lemma 
  \ref{lem:oshorig}) can be given which reduce the 
   determination of $\CR_{C''}(X_{n})$ in all cases to 
     inspection of Dynkin diagrams (alternatively, the results could be listed 
     as tables).  Moreover, the action of $G_{\s}$, 
     where $\s=\s_{X_{n}}$ can be explicitly determined
  (in most cases it is either trivial or given by Proposition 
  \ref{prop:diagaut}).
  
\subsection{} \label{ss:classical}We give some examples. Suppose that $\Phi$ is of type $B_{n}$ with $n\geq 3$. Enumerate the simple roots $\a_{1},\ldots, \a_{n}$ exactly as in \cite{Bour}, so $\a_{1}$ and $\a_{n}$ are terminal in the Coxeter graph and $\a_{n}$ is short. Denote the highest long (resp., short) root of $\Phi$ as $\a_{0}^{l}$ (resp., $\a_{0}^{s}$). We write, for example,  $\a_{m,\ldots, n}^{s}$ (resp., $\a_{m,\ldots, n}^{l}$) for the highest short (resp., long) root of the subsystem with simple roots
$\set{\a_{m},\ldots, \a_{n}}$. 
Then for $2\leq m\leq n$ one has 
\begin{equation*}
\CR_{C''}(\s_{B_{m}})=\set{(\a_{0}^{l},-\a_{2},\ldots, -\a_{m-1}, -\a_{m,\ldots, n}^{s})}
\end{equation*} 
  Proposition \ref{prop:genconj} implies there is a single conjugacy class of subsystems of type $B_{m}$.

For $4< m\leq n$, one has 
\begin{equation*}
\CR_{C''}(\s_{D_{m}})=\set{\bb_{m}:=(\a_{0}^{l},-\a_{2},\ldots, -\a_{m-1}, -\a_{m-1,\ldots, n}^{l})}.
\end{equation*} For $m=4<n$, one has 
$\CR_{C''}(\s_{D_{4}})=\set{\bb_{4}}\cup\set{\bd_{i}:=(\a_{0}
^{l}, -\a_{2},-\a_{i}, -\a_{4-i})\mid i=1,3}$. Observe that $[\bb_{4}],
[\bd_{1}]=[\bd_{3}]$ are simple systems for the same subsystem of $\Phi$ (of type $D_{4})$. Hence they 
are conjugate. We conclude there is a unique conjugacy class of simple systems of type $D_{m}$ in 
$\Phi$ if $4\leq m\leq n$. 

These results  may also be seen by (or used in) calculating 
 the  dot actions. For example, consider the dot action  of $G:=G_{\s_{D_{4}}}=\mpair{(1,3),(3,4)}
 \cong \Sym_{3}$ on $\CR_{C''}(\s_{D_{4}})$ where $4<n$. One obviously 
 has $(3,4)\cdot\bd_{1}=\bd_{3}$ (noting $(3,4)\bd_{1}=
 \bd_{3}$ and using Proposition \ref{prop:dots}(a)) and hence $(3,4)\cdot\bb_{4}=\bb_{4}$.   A more 
 tedious  calculation shows $(1,3)\cdot \bd_{3}=\bb_{4}$
  and that $(1,3)\cdot  \bd_{1}=\bd_{1}$. This can also be 
  seen more indirectly  as follows. By Propositions \ref{prop:dots} and 
  \ref{prop:diagaut}, $(1,3)\cdot \tau_{3}(\bu)=\tau_{3}(\bu)$ for all
   $\bu\in \CR_{C''}(\s_{D_{4}})$. This implies 
   $\set{\bd_{3},\bb_{4}}$ must be stable under the dot action 
   of $(1,3)$. If the restricted dot action of $(1,3)$ on this set is trivial, then
    $(1,3)$ would fix 
$\CR_{C''}(\s_{D_{4}})$ pointwise. But  from above,  
$G$ acts transitively on $\CR_{C''}(\s_{D_{4}})$, so $(1,3)$ 
can't act trivially.

Note that the dot action provides finer information than just the 
the description of conjugacy classes. For example, one may determine from  above which diagram automorphism of, say,  
$[\bd_{1}]=[\bd_{3}]$ are induced by the action of an element of $W$.

Although the above results on conjugacy in classical type are well known, the same techniques can be applied in general to determine even by hand  the conjugacy classes of irreducible subsystems for $\Phi$ of exceptional type.  We conclude by indicating how the methods can be extended to the study of possibly reducible subsystems.

\subsection{Standard genera of reducible type}
The next Lemma permits some reductions  in the study of conjugacy of  simple subsystems $\G$ of $\Phi$ to the case when both    $\Phi$ and $\G$ are irreducible.
\begin{lemma}\label{lem:simpcon}
 Let $\ba:=(\a_{1},\ldots, \a_{n})\in \Phi^{n}$ where $n\geq 1$. \begin{num}\item Let $\Psi$ be a 
 union  of  components of $\Phi$ such that  $[\ba]\seq \Psi$ and let $W':=W_{\Psi}$.
Then $\ba\in \CC_{W}^{(n)}$ if and only if $\ba\in \CC_{W'}^{(n)}$. 
\item  Suppose that  
  $1\leq m< n$ and that  $\mpair{\a_{i},\a_{j}}=0$ 
  for  all $1\leq i\leq m$ and $m+1 \leq j\leq n$.
  Suppose also that for each $1\leq j\leq m $, there is some $i<j$ with $\mpair{\b_{i},\b_{j}}\neq 0$.
 Let $W':=W_{\a,m}$, $\ba':=(\a_{1},\ldots,\a_{m})$  
and $\ba'':=(\a_{m+1},\ldots,\a_{n})$. 
Let $\Psi$ be the component of $\Phi$ containing $\a_{1}$. 
Then $\ba\in \CC_{W}^{(n)}$ if and only if $\ba'\in \CC_{W_{\Psi}}^{(m)}$ and $\ba''\in \CC_{W'}^{(n-m)}$. 
\item In $\text{\rm (b)}$,  $[\ba]$ is a simple subsystem of $\Phi$  if and only if $[\ba']$ is a simple 
subsystem of $\Psi$  and $[\ba'']$ is a  simple subsystem of $\Phi_{W''}$.
\end{num}
\end{lemma}
\begin{proof}
 Note that $W'_{\ba,i}=W'\cap W_{\ba,i}$ for $i=1,\ldots, n$.
 Since $\Psi\perp (\Phi\sm \Psi)$, it follows  that 
 $\CC_{W'_{\ba,i}}\cap \real \Psi=\CC_{W_{\ba,i}}\cap \real \Psi$. By the 
 definitions, this   implies (a).  Part (b) follows from \eqref{eq:fundchrec} 
 and (a), and (c) holds since $[\ba]=[\ba']\disjun [\ba'']$ with 
 $[\ba']\perp [\ba'']$.
 \end{proof}
\subsection{} \label{ss:red3} We refer to the genera  $\sigma_{A_n}=\s_{n}$ with $n\geq 1$,
$\s_{B_{n}}$ with $n\geq 3$,  $\s_{C_{n}}$ with $n\geq 2$, $\s_{D_{n}}$ with $n\geq 4$, $\s_{E_{n}}$ with $n=6$, $7$ or $8$, $\s_{F_{4}}$ or $\s_{G_{2}}$
as the standard irreducible genera. If $\s$ is one of these standard irreducible genera, of rank $n$, say, and 
$\bb=(\b_{1},\ldots, \b_{n})\in \CR_{C''}(\s)$  then 
$[\bb]$ is an irreducible simple system and hence it is  entirely contained in one component $\Psi$ of $\Phi$. In fact, $\bb\in \CS_{C''}(\s)\cap \Psi^{n}$.
The stabiliser of $\bb$ in $W$ is the standard parabolic subgroup
generated by $\Pi\cap[\bb]^{\perp}=(\Pi\cap \Psi\cap [\bb]^{\perp})\disjun (\Pi\sm \Psi)$, which may be readily determined by inspection of 
Dynkin diagrams  since (with two types of exceptions which may be easily dealt with), each root $\b\in [\bb]$ is, up to sign, a dominant root for an irreducible subsystem of $\Phi$.

 \subsection{} Given  genera $\r_{i}$  of ranks $n_{i}$ for $i=1,\ldots, k$, define the genus $\r=(\r_{1},\ldots, \r_{k})$ of rank $n=\sum_{i=1}^{k}n_{i}$ as follows.
   Let $N_{i}:=n_{1}+\ldots +n_{i}$ for $i=0,\ldots, k$,
 Say that $\bb:=(\b_{1},\ldots, \b_{n})\in \Phi^{n}$ is of genus $\r$ if
 \begin{conds}
 \item For $i=1,\ldots, k$, $(\b_{N_{i-1}+1},\ldots, \b_{N_{i}})$ is of genus $\r_{i}$.
 \item  For $1\leq i<j\leq k$ and $p,q\in \Nat$ with $N_{i-1}+1\leq p\leq N_{i}$ and $N_{j-1}+1\leq q\leq N_{j}$, one has  $\mpair{\b_{p},\b_{q}}=0$.
 \end{conds}
 
\subsection{}  Fix an arbitrary  total order $\leq$ of  the standard genera  (e.g. order them by decreasing 
 rank, with genera of equal rank ordered in the order in which they were listed 
 at the start of \ref{ss:red3}).  We say that a genus $\rho$ is standard
 if it is of the form $\rho=(\rho_{1},\ldots,\rho_{k})$ where each $\rho_{i}$ is a standard irreducible genus and $\rho_{1}\leq \rho_{2}\leq \ldots\leq \rho_{k}$.
 Then for every isomorphism type $X$ of crystallographic root system of rank $n$,
 there is a unique standard genus $\s_{X}$ of that type. Using Lemma \ref{lem:simpcon},
 for any root system $\Phi$, the elements $\bb$ of 
 $\CR_{C''}(\s_{X})$ can be readily determined.  Every 
 unordered simple subsystem of $\Phi$ of  type $X$ is 
 conjugate to one or more  of the subsystems $[\bb]$ for   
 $\bb\in\CR_{C''}(\s_{X})$.  The determination of the 
 conjugacy classes of simple systems of type $X$ then 
 reduces to the description of the 
  action of $G_{\s_{X}}$ on $\bb\in\CR_{C''}(\s_{X})$; this is  trivial in many cases (e.g. when $\vert\CR_{C''}(\s_{X})\vert=1$ or $\vert G_{\s_{X}}\vert =1$) but can require
  significant computational effort (if calculated by hand) in other situations. 

\subsection{Conclusion} The process outlined above provides an alternative procedure to the
 well known one
(see e.g. \cite{DyLeRef}) which is based on the results of  
Borel and de Siebenthal for determining the isomorphism
 classes of  unordered simple subsystems of $\Phi$.
However, the results here include additional information 
such as explicit representatives of the simple subsystems, 
which can be used in the study of more refined questions
 such as conjugacy of subsystems and whether an isomorphism between two subsystems can be realised as a 
 composite of a specified (diagram) automorphism of $\Phi$ with the action of an element of $W$.  
 The classification of conjugacy classes of simple  subsystems has been completed computationally; for 
instance, it can be found in \cite{Ro} (see also \cite{Osh}). In classical types,  even more 
refined information is readily accessible (e.g.   from the explicit list in \cite{DyRef} 
of all simple  subsystems contained in the positive 
roots together with standard descriptions of the actions 
of the groups involving (signed) permutations).  However, the techniques discussed in this paper provide a uniform conceptual approach,
applicable to all types. Moreover their geometric underpinning, our results on fundamental domains and stratification, have 
more general applicability.

%%%%%%%%%%%%%%%%APPENDIX IS BELOW%%%%%%%%%%%%%%%%%%%%%%%%%%%%%%%%%%%%%%%

 \bibliography{fundconjbibl}
\bibliographystyle{plain}
\end{document}